\documentclass[smallextended]{svjour3}       
\RequirePackage{fix-cm}

\usepackage{amsmath}
\usepackage{amssymb}
\usepackage{amsfonts}

\usepackage{bm}
\usepackage{amscd}
\usepackage{bbm,mathrsfs,dsfont}
\usepackage{graphicx}
\usepackage{subfigure}
\numberwithin{figure}{section}
\usepackage{cite}
\usepackage{latexsym}
\usepackage[figuresright]{rotating}
\usepackage[dvips,usenames]{color}
\usepackage[colorlinks=true,
            linkcolor=blue,
            citecolor=blue,
            urlcolor=blue,
            breaklinks=true]{hyperref}
\usepackage{enumitem}  

\numberwithin{equation}{section} 
\newtheorem{assumption}{Assumption}

\smartqed  
\begin{document}
\title{On Optimal Convergence Rates for the Nonlinear Schr\"{o}dinger Equation with a Wave Operator via Localized Orthogonal Decomposition}



\author{Hanzhang Hu \and
        Zetao Ma \and
        Lei Zhang \thanks{Corresponding author: \texttt{lzhang2012@sjtu.edu.cn}}
}

\titlerunning{On Optimal Convergence Rates for the NLSE with a Wave Operator via LOD}
\authorrunning{Hu et al.}

\institute{Hanzhang Hu \at
              School of Mathematics, Jiaying University, Meizhou 514015, Guangdong, China. \\
              \email{huhanzhang1016@163.com}           
           \and
           Zetao Ma \at
              School of Mathematical Sciences, Shanghai Jiao Tong University, 800 Dongchuan Road, Shanghai 200240, China. \\
              \email{770120068@qq.com}           
           \and
           Lei Zhang \at
              School of Mathematical Sciences, MOE-LSC and Institute of Natural Sciences, Shanghai Jiao Tong University, 800 Dongchuan Road, Shanghai 200240, China. \\
              \email{lzhang2012@sjtu.edu.cn}           
}

\date{Received: date / Accepted: date}

\maketitle

\begin{abstract}
In this paper, we develop a Localized Orthogonal Decomposition (LOD) method for the two-dimensional time-dependent nonlinear Schr\"{o}dinger equation with a wave operator. We prove that our method preserves conservation laws and admits a unique numerical solution; furthermore, we obtain unconditional (i.e., time-step restriction-free) optimal-order superconvergent \(L^p\) error estimates. 
To complement the theoretical analysis, we present a series of numerical simulations that verify the analytical results and further illustrate structural aspects of the problem.
\end{abstract}


\section{Introduction }
The nonlinear Schr\"{o}dinger (NLS) equation models a wide range of physical phenomena and has been extensively studied over the past several decades \cite{Changz-2003, Lix, Wang2015, Akrivis-1991, Jin-2015, Wu-2011, Wu-2012}, owing to its broad applications in areas such as plasma physics, nonlinear optics, water waves, biomolecular dynamics, and protein chemistry. In particular, when accounting for the nonlinear interaction of monochromatic waves, a NLS equation incorporating a wave operator was introduced in \cite{kelley1965self}. This model also arises naturally in the study of soliton propagation in plasmas. In this work, we consider the following initial-boundary value problem for the two-dimensional NLS equation with a wave operator:
 \begin{align}\label{1.1}
 \begin{split}
   &\partial^2_{t} u+i\partial_{t} u-{\rm div} (b(\bm{x})\nabla u)+V(\bm{x})u+f(|u|^2)u=0,\quad (\bm{x},t)\in \Omega \times(0,T],\\
  &u(\bm{x},0)=u_0(\bm{x}),\ \partial_t u(\bm{x},0)=u_1(\bm{x}),\quad \bm{x}\in\bar{\Omega}, \\
 &u(\bm{x},t)=0,\quad \bm{x}\in \partial\Omega,\quad 0\leq t\leq T,
 \end{split}
 \end{align}
in a convex polygonal domain $\Omega\subset \mathbb{R}^2$, $\partial\Omega$ is the boundary of $\Omega$, where $ i=\sqrt{-1}$ is the imaginary unit and $u:\Omega\rightarrow \mathbb{C} $ is the complex-valued unknown solution. 
The initial function \( u_0(\bm{x}) \) is assumed to be smooth. The coefficient \( b(\bm{x}) \) is a symmetric, uniformly positive definite matrix function, satisfying the ellipticity condition: there exist parameter-independent constants \( 0 < b_{*} < b^{*} < \infty \) such that for all \( \xi \in \mathbb{R}^2 \) and almost every \( \bm{x} \in \Omega \),
\begin{equation*}
    b_{*}|\xi|^2 \leq \xi^{T} b(\bm{x})\xi \leq b^{*}|\xi|^2.
\end{equation*}
Furthermore, the potential \( V(\bm{x}) \) is a given real-valued function. The nonlinearity \( f: \mathbb{R} \to \mathbb{R} \) is the derivative of a potential function \( F: \mathbb{R} \to \mathbb{R} \), and we assume \( f \in C^2(\mathbb{R}) \). A typical example is the power-type nonlinearity:
\begin{align} \label{nonlinearity-p}
f(s) = \pm s^{(p-1)/2} \quad \text{and} \quad F(s) = \pm \frac{2}{p+1} s^{(p+1)/2} \quad \text{for } p > 1,
\end{align}
which corresponds to \( f(|u|^2) = \pm |u|^{p-1} \) and \( F(|u|^2) = \pm \frac{1}{p+1} |u|^{p+1} \) when applied to \( s = |u|^2 \).

When $b(\bm{x})=1$ is a common model, computing the inner product of equation \eqref{1.1} with
$\partial_ t u$, and then taking the real part, the  conservative
law is obtained by
\begin{align}
  E(t)&=\int_\Omega \frac12\big[|\partial_ t u|^2+|\nabla u|^2+ V|u|^2+ F(|u|^2) \big]dx =E(0).
 \end{align}

Over the years, significant efforts have been devoted to developing high-efficiency and high-precision algorithms for two-dimensional NLS equations. Since nonconservative finite difference or spectral schemes have been shown to exhibit unphysical nonlinear blow-up under certain conditions \cite{Zhang-1995}, conservative schemes are generally preferred. Early contributions include the work of \cite{Akrivis-1991}, which introduced Galerkin finite element methods combined with Crank–Nicolson time discretization, achieving second-order accuracy. Subsequent developments include compact split-step finite difference schemes for constant and variable coefficient problems \cite{Dehghan}, alternating direction implicit (ADI) methods to reduce computational cost in multi-dimensional settings \cite{Gao, Mohebbi}, and conservative finite difference formulations that are unconditionally stable and convergent with second-order accuracy in the maximum norm \cite{HuChen2016}. Despite these advances, challenges remain in balancing computational efficiency, algorithmic robustness, and adaptability to strongly nonlinear regimes.

To maintain conservation properties and avoid time-step restrictions, implicit Crank–Nicolson-type temporal discretizations—which offer second-order accuracy—are often preferred for solving \eqref{1.1}. However, the resulting discrete system for the two-dimensional nonlinear Schr\"{o}dinger equation is nonlinear and typically requires iterative solution methods, leading to high computational costs. It is therefore essential to design discrete spaces that respect the inherent structure of the problem, particularly to preserve conservation laws accurately. Localized Orthogonal Decomposition (LOD) method is a highly effective approach in this context. Originally introduced in \cite{maalqvist2014} (see also \cite{zhang_RPS, owhadi2017multigrid, altmann2021numerical, book_Peterseim, liuGRPS, hauck2023super}) for elliptic problems with heterogeneous and highly varying (non-periodic) coefficients, LOD provides a framework for embedding structural features of the differential operator directly into the discrete space. Its high-efficiency capacity for resolving fine-scale structures has established it as an influential technique, attracting broad interest due to its wide applicability to diverse multiscale and nonlinear problems \cite{gallistl2018numerical,henning2020computational,liu2021iterated,henning2022,maier2022multiscale,verfurth2022numerical,henning2023optimal,peterseim2024super,refId0}.

The core idea of the LOD method is to decompose the high-dimensional solution space using the energy inner product associated with the problem. This decomposition splits the space into a fine-scale subspace and its orthogonal complement—a coarse space endowed with excellent approximation properties. A key feature of LOD is its local construction: the low-dimensional coarse space is spanned by basis functions with local support, each obtained by solving patch-wise problems on small subdomains. These local problems are computationally inexpensive to solve and can be processed independently, making the method highly amenable to parallelization.

Interestingly, LOD also delivers significantly improved accuracy for single-scale problems with high regularity, as demonstrated in \cite{henning2022}. Building on this accuracy enhancement, the present paper develops an LOD-based multiscale method for solving the two-dimensional NLS equation with the wave operator.

The main contributions of this work are threefold: we establish the existence and uniqueness of solutions to the proposed discrete system, prove conservation properties, and derive optimal superconvergent $L^p$-error estimates without imposing any time-step restrictions. In contrast to the approaches in \cite{henning2022}, the existence, uniqueness, and $L^\infty$-boundedness of the finite element solutions for multi-dimensional settings are proved using Schaefer’s fixed point theorem. A novel theoretical achievement of this paper is the proof of a superconvergence rate of order $O(\tau^2 + H^4)$ in the $L^p$-norm for the proposed numerical scheme.

The remainder of this paper is organized as follows. 
Section \ref{sec:Preliminaries} introduces the preliminary concepts, key properties of the LOD method, and the Crank–Nicolson LOD scheme for the NLS with wave operator. In this section, we also present our main results regarding the corresponding approximation properties.
Next, Section \ref{sec:Numerical experiments} presents the numerical experiments conducted to validate the theoretical findings, along with a discussion of the results. Finally, Section \ref{sec:Proof_of_Main_Theorems} provides the detailed proofs of the main theorems.

\section{Preliminaries and LOD method}
\label{sec:Preliminaries}
\subsection{Some Notations and Useful Lemmas}
Now, we introduce the complex Sobolev spaces \( W^{m,p}(\Omega) \) on the domain \( \Omega \), equipped with the norm \( \|\cdot\|_{m,p} \). For notational simplicity, we define:
\[
H^m(\Omega) = W^{m,2}(\Omega), \quad \|\cdot\|_m = \|\cdot\|_{m,2}, \quad \|\cdot\|_\infty = \|\cdot\|_{L^\infty}, \quad \|\cdot\| = \|\cdot\|_{0,2}.
\]
Additionally, we define the space
\[
H^1_0(\Omega) = \left\{ v \in H^1(\Omega) \mid v = 0 \text{ on } \partial\Omega \right\},
\]
and denote by \( |\cdot|_1 \) the semi-norm in \( H^1(\Omega) \).

Let \( (\cdot, \cdot) \) denote the \( L^2 \)-inner product:
\[
(u, v) = \int_\Omega u(x) \overline{v(x)} \, dx, \quad \text{for } u, v \in L^2(\Omega),
\]
with the corresponding norm \( \|\cdot\| \).

Besides, we now state a discrete Gronwall's inequality that will be used in the subsequent analysis.
\begin{lemma}\label{lem2.2}
{(Gronwall's inequality \cite{Zhou})} Suppose that the discrete mesh function $\{w_n,n=1,2,\cdots,N\}$, $N\tau=T$, satisfies the inequality
\begin{equation*}
w_n\leq A+ \tau \sum^{n}_{k=1}B_k w_k,\quad 0\leq n\tau \leq T,
\end{equation*}
where $A$ and $B_k\ (k=1,2,\cdots,N)$ are nonnegative constants. Then,
\begin{equation}
\max_{1\leq n \leq N}w_n \leq Ae^{2 \tau \sum^{N}_{k=1}B_k},
\end{equation}
where $\tau$ is sufficiently small, such that $\tau(\underset{k=1,2,\cdots,N}{\max} B_k)\leq \frac{1}{2}$ .
\end{lemma}

The superconvergence result established in this paper relies on the choice of a suitable generalized finite element space for discretizing the NLS equation with wave operator. Specifically, we employ LOD space. In what follows, we provide a brief overview of the LOD methodology in a general context pertinent to our study and recall several key results that will be essential for the subsequent error analysis.

\subsection{Ideal LOD Space and Approximation Properties}

Let \( a(\cdot, \cdot) \) be an inner product defined on the Sobolev space \( H^1_0(\Omega) \). Given a source term \( f \in H^1_0(\Omega) \cap H^2(\Omega) \), we seek a function \( u \in H^1_0(\Omega) \) satisfying the variational equation:
\[
a(u, v) = (f, v) \quad \text{for all } v \in H^1_0(\Omega).
\]
The existence and uniqueness of the solution \( u \) are guaranteed by the Riesz representation theorem.

The LOD method constructs a low-dimensional discrete subspace designed to approximate \( u \) with high accuracy. The construction starts with a coarse space \( V_H \subset H^1_0(\Omega) \), defined as the standard $\mathbb{P}^1$ Lagrange finite element space over a quasi-uniform simplicial mesh of the domain \( \Omega \). Here, \( H \) denotes the mesh size, and \( \mathcal{T}_H \) represents the corresponding simplicial partition.

It is well known that if \( u \in H^2(\Omega) \), the Galerkin approximation \( u_H \in V_H \) converges optimally in the following sense:
\[
\|u - u_H\| + H \|u - u_H\|_{1} \leq C H^2 \|u\|_{H^2(\Omega)},
\]
where \( C \) is a generic positive constant that depends only on the mesh regularity.

A natural question arises: does there exist a low-dimensional subspace of \( H^1_0(\Omega) \) with the same dimension as \( V_H \), but offering significantly improved approximation properties? To explore this, we aim to enrich \( V_H \) by incorporating information from the underlying differential operator.

We begin by introducing the \( L^2 \)-projection \( P_H : L^2(\Omega) \to V_H \), defined by
\[
(P_H w, v_H) = (w, v_H) \quad \text{for all } w \in L^2(\Omega) \text{ and } v_H \in V_H.
\]
On quasi-uniform meshes, this projection is known to be \( H^1 \)-stable (see, e.g., \cite{Evans-PDE}). The kernel of \( P_H \) restricted to \( H^1_0(\Omega) \), given by
\[
W := \ker(P_H) = \{ w \in H^1_0(\Omega) \mid P_H w = 0 \},
\]
forms a closed subspace of \( H^1_0(\Omega) \), referred to as the fine space. This leads to the ideal \( L^2 \)-orthogonal decomposition:
\[
H^1_0(\Omega) = V_H \oplus W.
\]

To enhance \( V_H \) with fine-scale information while preserving its low dimensionality, we define the \( a(\cdot, \cdot) \)-orthogonal complement of \( W \):
\[
V_{\mathrm{LOD}} = \{ v \in H^1_0(\Omega) \mid a(v, w) = 0 \text{ for all } w \in W \}.
\]
By construction, \( \dim(V_{\mathrm{LOD}}) = \dim(V_H) = N_H \). This yields another ideal decomposition of \( H^1_0(\Omega) \):
\[
H^1_0(\Omega) = V_{\mathrm{LOD}} \oplus_a W,
\]
where \( V_{\mathrm{LOD}} \) and \( W \) are orthogonal with respect to \( a(\cdot, \cdot) \).

To assess the approximation quality of \( V_{\mathrm{LOD}} \), we define the LOD projection \( R_H u \in V_{\mathrm{LOD}} \) of \( u \in H^1_0(\Omega) \) via
\begin{equation}\label{lodproj}
a(R_Hu, v) = a(u, v) \quad \text{for}\  \forall \  v \in V_{\mathrm{LOD}}.
\end{equation}
It follows that \( a(R_H u - u, v) = 0 \) for all \( v \in V_{\mathrm{LOD}} \), implying \( R_H u - u \in W \) due to the \( a \)-orthogonality of the decomposition.

An important consequence of this construction is the following orthogonality relation, referred to as LOD-orthogonality: for any \( w \in W \),
\begin{equation}\label{l2.3}
a(u - R_Hu, w) = (g, w).
\end{equation}
As shown in \cite{henning2022}, the following estimates hold:
\[
\|u - R_H u\| + H \|u - R_H u\|_1 \leq C H^4 \|g\|_{H^2(\Omega)},
\]
\[
\|R_H u\|_{H^1} \leq C \|g\|,
\]
where \( C > 0 \) is a constant depending on \( f \) and the coercivity constant of \( a(\cdot, \cdot) \).

\subsection{Localization of the Orthogonal Decomposition}

In practice, the full LOD space \(V_{\mathrm{LOD}}\) is computationally inefficient because its basis functions have global support. This leads to high computational costs in constructing the basis and results in dense stiffness matrices in Galerkin discretizations. Fortunately, these basis functions are known to exhibit exponential decay outside small nodal neighborhoods, making it possible to accurately approximate them using localized functions. Below, we summarize the localization strategy introduced and analyzed in [24, 30] for efficiently approximating \(V_{\mathrm{LOD}}\).

Let \( \ell \in \mathbb{N}^+ \) be a localization parameter that controls the support size of the resulting basis functions (of order \( \mathcal{O}(\ell H) \)). For any simplex \( K \in \mathcal{T}_H \), define the \( \ell \)th-layer patch around \( K \) iteratively as:
\[
S_{\ell}(K) := \bigcup \left\{ T \in \mathcal{T}_H \mid T \cap S_{\ell-1}(K) \neq \emptyset \right\}, \quad \text{with } S_0(K) := K.
\]
That is, \( S_{\ell}(K) \) consists of \( K \) and \( \ell \) layers of adjacent grid elements. The restriction of the fine space \( W = \ker(P_H) \) to this patch is defined as:
\[
W(S_{\ell}(K)) := \left\{ w \in H^1_0(S_{\ell}(K)) \mid P_H w = 0 \right\} \subset W.
\]

To construct a nearly \( a(\cdot, \cdot) \)-orthogonal correction to a given coarse function \( v_H \in V_H \), we proceed as follows. For each \( K \in \mathcal{T}_H \) such that \( K \subset \text{supp}(v_H) \), find a local correction \( Q_{K,\ell} v_H \in W(S_{\ell}(K)) \) satisfying:
\begin{equation}\label{locproj}
a(Q_{K,\ell}v_H, w) = -a_{K}(v_H, w) \quad \text{for all } w \in W(S_{\ell}(K)),
\end{equation}
where \( a_K(\cdot, \cdot) \) denotes the restriction of the bilinear form \( a(\cdot, \cdot) \) to the element \( K \). Since this problem is confined to the patch \( S_{\ell}(K) \), it is local and computationally efficient to solve. The corrected function is then defined as:
\[
R_{\ell} v_H := v_H + \sum_{K \in \mathcal{T}_H} Q_{K,\ell} v_H.
\]

In practice, the corrected function \( R_{\ell}v_H \) is computed for each nodal basis function of \( V_H \). The LOD space, which approximates the ideal space \(V_{\mathrm{LOD}}\), is then defined as:
\begin{equation}\label{l2.9}
V_{\ell,\mathrm{LOD}} := \{R_{\ell}v_H \mid v_H \in V_H\}.
\end{equation}
Note that if \( \ell \) is chosen large enough so that \( S_{\ell}(K) = \Omega \), then by equation \eqref{locproj} we have:
\[
a(R_{\infty}v_H, w) = \sum_{K \in \mathcal{T}_H} \bigl( a_K(v_H, w) + a(Q_{K,\ell}v_H, w) \bigr) = 0 \quad \text{for all } w \in W.
\]
Hence, the functions \( R_{\infty}v_H \) span the \( a(\cdot, \cdot) \)-orthogonal complement of \( W \); that is, they exactly span the ideal space \(V_{\mathrm{LOD}}\). For smaller values of \( \ell \), it is natural to ask how well \( V_{\ell,\mathrm{LOD}} \) approximates \(V_{\mathrm{LOD}}\). The following lemma from \cite{henning2022} addresses this.

\begin{lemma}
Under the general assumptions of this section and assuming \( g \in H^1_0(\Omega) \cap H^2(\Omega) \), let \( V_{\ell,\mathrm{LOD}} \) be defined as in \eqref{l2.9}, and let \( R_{\ell}u \in V_{\ell,\mathrm{LOD}} \) be the Galerkin approximation of \( u \), i.e., the solution to
\[
a(R_{\ell}u, v) = (g, v) \quad \text{for all } v \in V_{\ell,\mathrm{LOD}}.
\]
There exists a generic constant \( \rho > 0 \) (depending on \( a(\cdot, \cdot) \), but independent of \( \ell \) and \( H \)) such that:
\begin{align}\label{l2.10}
\begin{split}
\|u - R_{\ell}u\| &\leq C \left(H^4 + e^{-\rho \ell}\right) \|g\|_{H^2(\Omega)}, \\
\|u - R_{\ell}u\|_{H^1(\Omega)} &\leq C \left(H^3 + e^{-\rho \ell}\right) \|g\|_{H^2(\Omega)}.
\end{split}
\end{align}
Here, \( C > 0 \) depends on the coercivity and continuity constants of \( a(\cdot, \cdot) \), as well as on \( \Omega \), but is independent of \( \ell \), \( H \), and \( u \).
\label{lem:lod}
\end{lemma}

Choosing \( \ell \geq 4|\log H| / \rho \) ensures that the optimal convergence rates—\( \mathcal{O}(H^4) \) in the \( L^2 \)-norm and \( \mathcal{O}(H^3) \) in the \( H^1 \)-norm—are preserved. Although \( \rho \) is generally unknown, numerical experience and literature (e.g., \cite{henning2022}) show that small values of \( \ell \) are often sufficient to achieve optimal accuracy relative to the mesh size \( H \). For the subsequent error analysis, we will therefore work within the ideal LOD framework of Section 2.1, effectively neglecting the effects of localization and disregarding the exponentially small truncation error. This is equivalent to working with the localized LOD space with \(\ell\) chosen sufficiently large so that Lemma \ref{lem:lod} holds.

\begin{remark}
    The estimates in Lemma \ref{lem:lod} can be further refined. For example, the exponential term often depends only on the \( L^2 \)-norm of \( g \), rather than its full \( H^2 \)-norm. Moreover, the decay rate for the \( L^2 \)-error is generally faster than that for the \( H^1 \)-error.
\end{remark}


\subsection{The Crank--Nicolson LOD Scheme}
In this paper, we make the following regularity assumptions on the solution of \eqref{1.1}:
\begin{assumption}[Regularity Assumptions]\label{ass:regularity}
We assume the following regularity conditions hold:
\begin{enumerate}[label=(\Roman*), leftmargin=*, itemsep=0.5ex]
    \item $u \in H^4(0,T; W_0^{1,p}(\Omega) \cap W^{2,p}(\Omega))$ for $2 \leq p < \infty$;
    \item $u_0 \in W_0^{1,p}(\Omega) \cap W^{2,p}(\Omega)$ with $\Delta u_0 \in W_0^{1,p}(\Omega)$;
    \item $V \in W^{2,p}(\Omega)$;
    \item $\partial_t^3 u,\ \partial_t^4 u \in L^2(0,T; L^2(\Omega))$;
    \item \label{ass:time_regularity} For optimal convergence rates in the error analysis (when $p=3$ in \eqref{nonlinearity-p}), we additionally require
    \[
    \partial_t u \in L^2(0,T;W^{4,p}(\Omega)), \qquad \partial_t^2 u \in L^2(0,T;W^{4,p}(\Omega)).
    \]
\end{enumerate}
\end{assumption}

The weak formulation of problem \eqref{1.1} is to find \( u \in H^1_0(\Omega) \) such that
\begin{align}\label{u}
  (\partial^2_t u, \varphi) + i (\partial_t u, \varphi) + (b\nabla u, \nabla \varphi) + (V u, \varphi) + (f(|u|^2) u, \varphi) = 0, \quad \forall \varphi \in H^1_0(\Omega).
\end{align}

Let \( R_H : H^1_0(\Omega) \to V_{\mathrm{LOD}} \) be the elliptic projection operator defined by
\begin{align}\label{uproj}
  a(R_H u - u, \varphi) = (b\nabla (R_H u - u), \nabla \varphi) + (V (R_H u - u), \varphi) = 0, \quad \forall \varphi \in V_{\mathrm{LOD}}.
\end{align}

We begin by introducing the time discretization. Let \( \tau = T/N \) denote the time step size and define the temporal nodes \( t^n = n\tau \) with the corresponding function values \( \phi^n = \phi(\cdot, t^n) \). The following finite difference operators and averaging operators will be used:
\begin{align*}
  \delta_{\hat{t}} \phi^n &= \frac{\phi^{n+1} - \phi^{n-1}}{2\tau}, &
  \delta_t \phi^n &= \frac{\phi^{n+1} - \phi^n}{\tau}, \\
  \delta_{\bar{t}} \phi^n &= \frac{\phi^n - \phi^{n-1}}{\tau}, &
  \delta^2_t \phi^n &= \frac{\phi^{n+1} - 2\phi^n + \phi^{n-1}}{\tau^2}, \\
  \bar{\phi}^n &= \frac{\phi^{n-1} + \phi^{n+1}}{2}, &
  \phi^0 &= \frac{\phi^{-1} + \phi^1}{2}.
\end{align*}
For the error analysis, we define the following \( L^2 \)-norms:
\begin{align*}
  \|\phi\|_{L^2(0,T;L^2)} &= \left( \int_0^T \|\phi(\cdot,t)\|^2  dt \right)^{1/2}, \\
  \|\phi\|_{L^2(t^{n-1},t^n;L^2)} &= \left( \int_{t^{n-1}}^{t^n} \|\phi(\cdot,t)\|^2  dt \right)^{1/2}.
\end{align*}

We now present the fully discrete LOD method for \eqref{1.1}. The scheme is formulated as follows: for \( n = 1, \dots, N \), find \( u^n_H \in V_{{LOD}} \) such that
\begin{align}\label{uh}
  (\delta^2_t u^n_H, \varphi) + i (\delta_{\hat{t}} u^n_H, \varphi) + (b\nabla \bar{u}^n_H, \nabla \varphi) + (V \bar{u}^n_H, \varphi) + (\tilde{f}(|u^{n+1}_H|^2, |u^{n-1}_H|^2) \bar{u}^n_H, \varphi) = 0
\end{align}
holds for all \( \varphi \in V_{\mathrm{LOD}} \). The scheme is initialized by setting
\[
u^0_H = u_{\mathrm{LOD}}(x,0) \quad \text{and} \quad u^{-1}_H = u^1_H - 2\tau u_{1,H},
\]
where \( u_{1,H} \) is a suitable approximation of \( \partial_t u(x,0) \). Here, the nonlinear term employs the averaged function
\[
\tilde{f}(x, y) = \frac{F(x) - F(y)}{x - y} = \int_0^1 f((1 - s)x + s y)  ds, \quad x, y \in \mathbb{R}^+.
\]

Here, we present the main results of the proposed numerical scheme.
\begin{theorem}\label{main_thm_2.1}
Suppose that \( u^n \) is the exact solution of \eqref{u} at \( t = t^n \), and \( u^n_H \) is the numerical solution satisfying the discrete equations \eqref{uh}. If the initial data are chosen such that \( R_H u_0 = u^0_H \), then for \( 2 \leq p < \infty \), the following error estimate holds
\begin{equation}
 \max_{1\leq n \leq N}\| u^n-u^n_H \|_{L^p}\leq C(\tau^2+H^{4}),
\end{equation}
where the constant $C > 0$ is independent of the mesh size $H$ and the time step $\tau$, but may depend on $\Omega$, $b$, and the regularity assumptions stated in Assumption \ref{ass:regularity}. 
\end{theorem}

The remainder of this section is organized as follows. 
First, in Theorem~\ref{thm:conservation_property} we prove that the proposed numerical scheme preserves a discrete energy. 
Second, Theorem~\ref{thm2.2} establishes the uniform boundedness of the numerical solution $\{u^n_H\}$ in the $L^\infty$-norm. 
Finally, in Section~\ref{sec:Proof_of_Main_Theorems} we combine these results to complete the proof of Theorem~\ref{main_thm_2.1}.

Now, we present the conservation property of the proposed numerical scheme.
\begin{theorem}\label{thm:conservation_property}
Assume that \( \tau \) is sufficiently small. Then, the solution \( u^n_H \) of the scheme for \( b(\bm{x})=1\  and \  0 \leq n \leq N \) satisfies the following energy conservation law:
\begin{equation} \label{2.4}
\begin{split}
E^n &:= \|\delta_t u^n_H\|^2 + \frac{1}{2} \left( |u^{n+1}_H|^2_1 + |u^n_H|^2_1 \right) +
\frac{1}{2}\int_{\Omega} V \bigl( |u^{n+1}_H|^2 + |u^n_H|^2 \bigr) \,\mathrm{d}x \\
&\quad + \frac{1}{2} \left( F(|u^n_H|^2) + F(|u^{n-1}_H|^2) \right) = E^0.
\end{split}
\end{equation}


\end{theorem}


\subsection{Existence, Uniqueness and $L^{\infty}$ Boundedness of LOD Solutions}
We now establish the existence, uniqueness, and boundedness of the numerical solution $u^n_H$ in $L^{\infty}$-norm.
\begin{theorem}\label{thm2.2}
 Assuming $\tau$ and $H$ are sufficiently small, the fully discretized problem \eqref{uh} has a unique solution $u^n_H$ that satisfies
\begin{equation}
 \max_{1\leq n \leq N}\|u^n_H\|_{L^{\infty}}\leq C.
\end{equation}

\end{theorem}

\section{Numerical Experiments }
\label{sec:Numerical experiments}
In this section, we present numerical experiments to verify the convergence rates established in the previous sections. 
For Example~\ref{Example1}, which admits a smooth exact solution, we directly compute errors against the analytical solution to confirm the optimal convergence rate and conservation properties. 
In Examples~\ref{Example2}--\ref{Example5}, where no closed-form exact solution is available, a reference solution is computed on a sufficiently fine mesh and used as a surrogate to evaluate the errors and assess the optimal convergence rate.


In the following, we present two benchmark examples (Examples~\ref{Example1} and~\ref{Example2}) designed to isolate the effect of coefficient heterogeneity:
\begin{itemize}
  \item \textbf{Homogeneous benchmark} (Example~\ref{Example1}): constant coefficient $b(x,y) \equiv 1$;
  \item \textbf{Inhomogeneous benchmark} (Example~\ref{Example2}): spatially varying quartic coefficient $b(x,y) = \bigl[(2.8 + x^2)(2.8 + y^2)\bigr]^2$.
\end{itemize}
Both examples share the same weakly perturbed harmonic potential $V(x,y)$, which provides a smooth setting for method validation.

\begin{example}\label{Example1}
In this numerical example, we consider two-dimensional NLS equation with wave operator, constant coefficient $b(x,y) \equiv 1$, and a smooth potential $V(x,y)$. The computational domain is the unit square $\Omega = [0,1]^2$ and the final time is $T = 1.0$:
\begin{subequations}\label{eq:nls-homogeneous}
\begin{align}
  &\partial_t^2 u + i \partial_t u - \Delta u + V u + |u|^2 u = 0, 
    && (x,y) \in \Omega,\ t \in [0,T], \label{eq:nls-pde} \\
  &u(x,y,t) = 0, 
    && (x,y) \in \partial\Omega,\ t \in [0,T], \label{eq:bc} \\
  &u(x,y,0) = \frac{1}{10} \sin(\pi x) \sin(\pi y), 
    && (x,y) \in \Omega, \label{eq:ic1} \\
  &\partial_t u(x,y,0) = -\frac{i}{10} \sin(\pi x) \sin(\pi y), 
    && (x,y) \in \Omega, \label{eq:ic2}
\end{align}
\end{subequations}
where the smooth potential is given by
\begin{equation}
V(x,y) = -2\pi^2 - \frac{1}{100} \sin^2(\pi x) \sin^2(\pi y). \label{eq:V-hom}
\end{equation}
The exact solution of the above initial-boundary value problem is
\begin{equation}
u(x,y,t) = \frac{1}{10} \sin(\pi x) \sin(\pi y) \, e^{-i t}. \label{eq:exact-hom}
\end{equation}


To demonstrate the spatial accuracy of the proposed algorithm, we fix the time step size at $  \tau = 1.0 \times 10^{-3}  $. The numerical results reported in Table~\ref{tab:lod_spatial_errors_1} show that the LOD method achieves approximately fourth-order convergence in both the $  L^2  $ and $  L^4  $ norms with respect to the spatial mesh size $  H  $, for $  H = 1/2, 1/4, 1/8, 1/16  $. For temporal accuracy, we adopt the relation $  \tau = H^2  $ and observe from Table~\ref{tab:lod_temporal_errors_1} that the method exhibits second-order convergence in both the $  L^2  $ and $  L^4  $ norms for $  \tau = 1/4, 1/16, 1/64, 1/256  $. Hence, these results confirm the theoretical convergence rates established in Theorem~\ref{main_thm_2.1}.


\begin{table}[h!]
\centering
\caption{\(L^2\) and \(L^4\) errors and spatial convergence rates for the LOD method (\(\tau = 1/1000\)).}
\label{tab:lod_spatial_errors_1}
\vspace{0.2cm}
\begin{tabular}{|c|c|c|c|c|}
\hline
\(H\) & \(\|u^n - u^n_H\|_{L^2}\) & Rate & \(\|u^n - u^n_H\|_{L^4}\) & Rate \\
\hline
\(1/2\)  & \(6.7403 \times 10^{-3}\) & —    & \(1.9671 \times 10^{-2}\) & —    \\
\hline
\(1/4\)  & \(4.7559 \times 10^{-4}\) & 3.83  & \(1.2596 \times 10^{-3}\) & 3.97  \\
\hline
\(1/8\)  & \(3.7269 \times 10^{-5}\) & 3.67  & \(9.1335 \times 10^{-5}\) & 3.79  \\
\hline
\(1/16\) & \(3.2311 \times 10^{-6}\) & 3.53  & \(8.0462 \times 10^{-6}\) & 3.50  \\
\hline
\end{tabular}
\end{table}

\begin{table}[h!]
\centering
\caption{\(L^2\) and \(L^4\) errors and temporal convergence rates for the LOD method.}
\label{tab:lod_temporal_errors_1}
\vspace{0.2cm}
\begin{tabular}{|c|c|c|c|c|c|}
\hline
\(H\) & \(\tau\) & \(\|u^n - u^n_H\|_{L^2}\) & Rate & \(\|u^n - u^n_H\|_{L^4}\) & Rate \\
\hline
\(1/2\)  & \(1/4\)    & \(6.0754 \times 10^{-3}\) & —    & \(1.6584 \times 10^{-2}\) & —    \\
\hline
\(1/4\)  & \(1/16\)   & \(4.6204 \times 10^{-4}\) & 1.91  & \(1.2055 \times 10^{-3}\) & 2.00  \\
\hline
\(1/8\)  & \(1/64\)   & \(3.7241 \times 10^{-5}\) & 1.83  & \(9.1018 \times 10^{-5}\) & 1.84  \\
\hline
\(1/16\) & \(1/256\)  & \(3.2113 \times 10^{-6}\) & 1.77  & \(8.0088 \times 10^{-6}\) & 1.76  \\
\hline
\end{tabular}
\end{table}

\begin{figure}[!hbt]
\centering
\begin{minipage}[t]{0.48\textwidth}
\centering
\includegraphics[width=\linewidth]{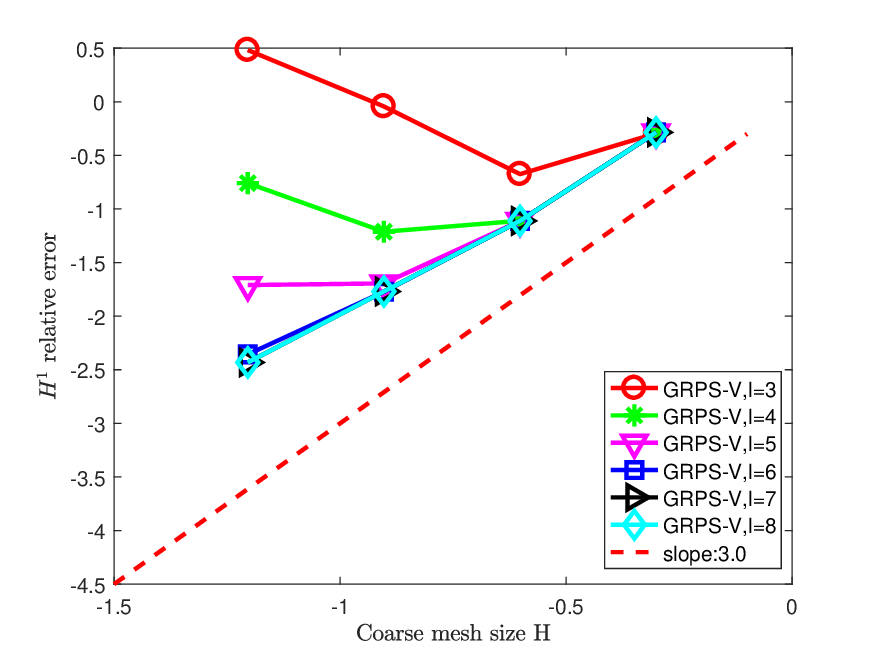}
\caption{Relative \(H^1\) error order. }
\label{fig:H1}
\end{minipage}
\hfill
\begin{minipage}[t]{0.48\textwidth}
\centering
\includegraphics[width=\linewidth]{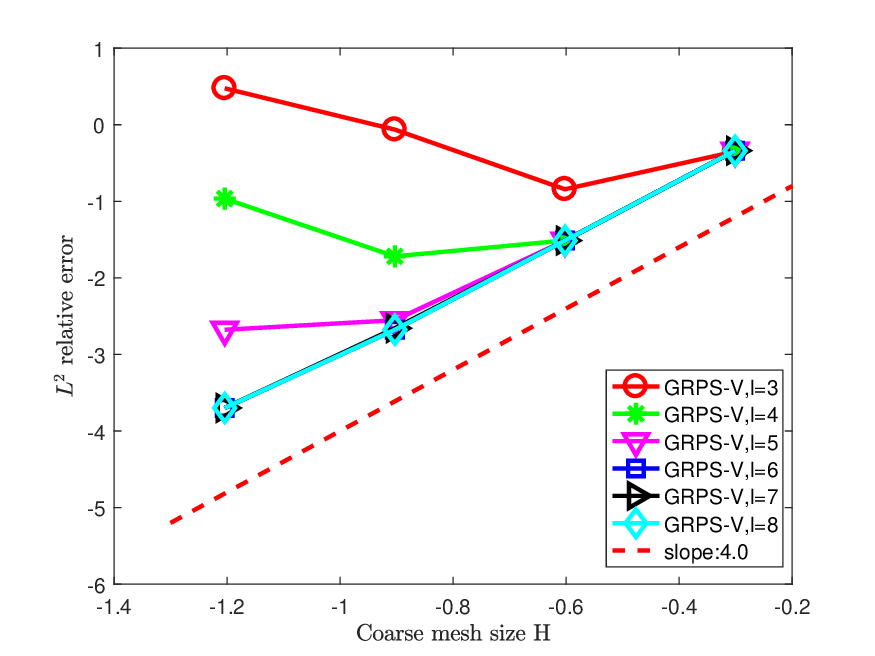}
\caption{Relative \(L^2\) error order.}
\label{fig:L2}
\end{minipage}
\end{figure}


Furthermore, the relative \(H^1\) and \(L^2\) errors are compared against the coarse mesh size \(H\) in Fig.~\ref{fig:H1} and Fig.~\ref{fig:L2}, respectively. The reference solution is computed on a fine mesh with \(h = 1/256\), while the coarse grids correspond to \(N_c = 2, 4, 8, 16\). Both axes are plotted on a logarithmic scale: the horizontal axis denotes the coarse mesh size \(H\), and the vertical axis shows the relative error \(\|u^{N} - u_H^{l,N}\|/\|u^{N}\|\).
These results were obtained using the GRPS-V multiscale basis with localization levels \(l = 3, 4, \dots, 8\), which yields approximately third-order convergence in the $  H^1  $ norm and fourth-order convergence in the $  L^2  $ norm.

\begin{figure}[!hbt]
\centering
\includegraphics[width=0.6\textwidth]{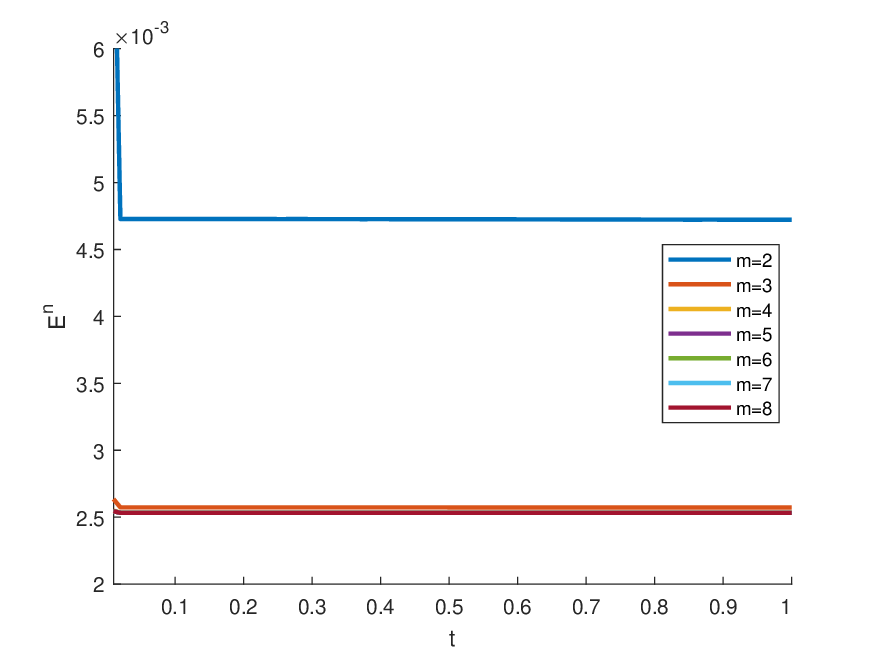}
\caption{Energy conservation \(E^n\) for \(m = 2, 3, 4, 5, 6, 7, 8\), with \(T=1.0\), \(h=1/64\), and \(\tau = 1.0 \times 10^{-2}\).}
\label{fig:energy}
\end{figure}

Finally, we consider a coarse mesh with size $  H = 1/4  $ and time step $  \tau = 1.0 \times 10^{-2}  $. As shown in Fig.~\ref{fig:energy}, the discrete energy $  E^n  $ remains nearly constant for localization layers $  m > 3  $, indicating excellent energy conservation. Figs.~\ref{fig:errE} and~\ref{fig:errE1} further illustrate the energy conservation errors by plotting $  E^n - E^0  $ and $  E^n - E  $, respectively. These numerical results confirm the perfect approximation property of the LOD method, which is in excellent agreement with Theorem~\ref{thm:conservation_property}.

\begin{figure}[!hbt]
\centering
\includegraphics[width=8cm]{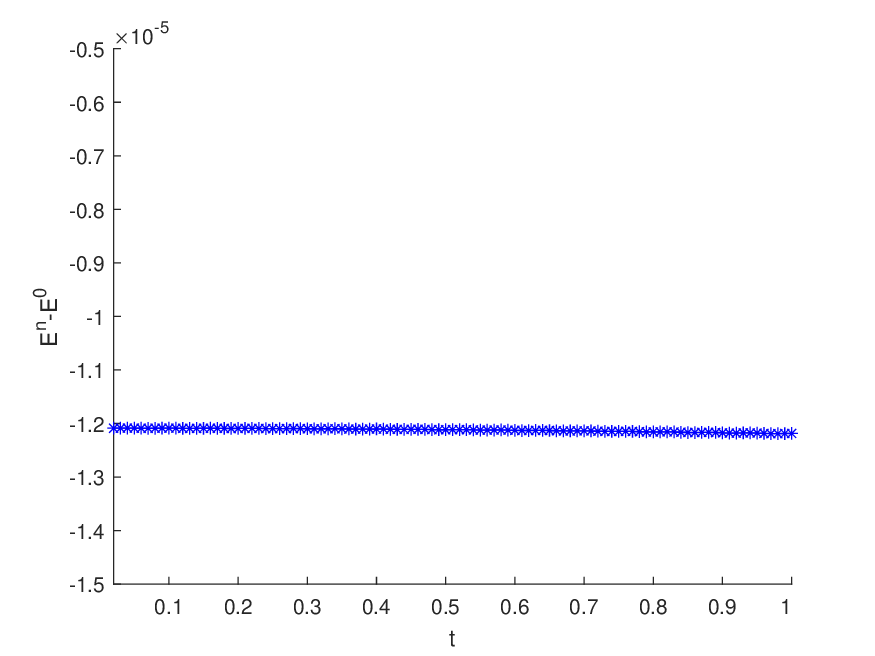}
\caption{Error in energy conservation, \(E^n - E^{0}\), with \(T=1.0\), \(h=1/64\), and \(\tau = 1.0 \times 10^{-2}\).}
\label{fig:errE}
\end{figure}

\begin{figure}[!hbt]
\centering
\includegraphics[width=8cm]{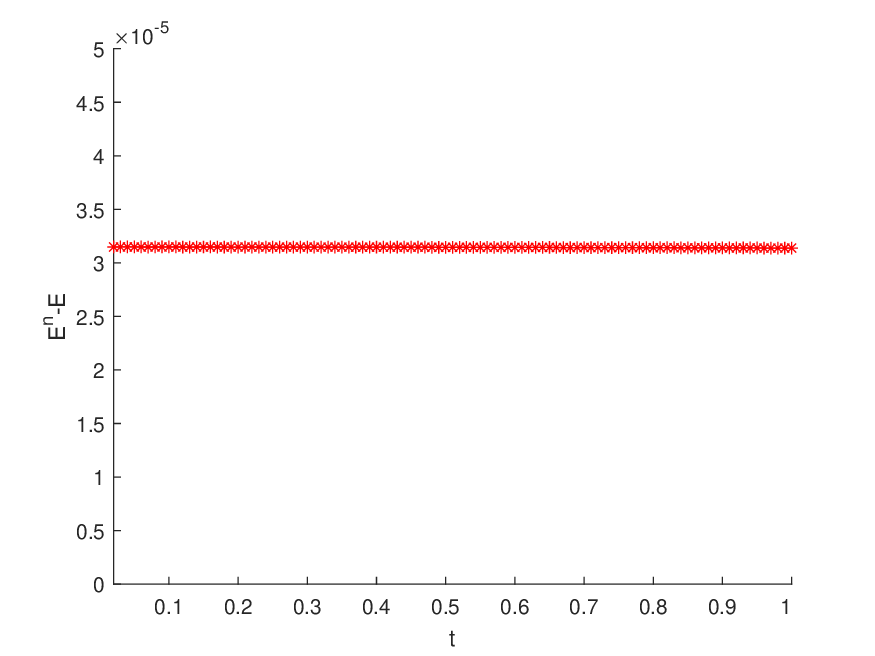}
\caption{Error in energy conservation, \(E^n - E\), with \(T=1.0\), \(h=1/64\), and \(\tau = 1.0 \times 10^{-2}\).}
\label{fig:errE1}
\end{figure}

\end{example}

\begin{example}\label{Example2}
In this example, we consider the two-dimensional NLS equation with wave operator, featuring a spatially varying quartic coefficient $b(x,y)$ and a smooth potential $V(x,y)$. The computational domain $\Omega$ and final time $T$ are the same as in Example~\ref{Example1}:
\begin{subequations}\label{eq:nls-inhom}
\begin{align}
  & \partial_t^2 u + i \partial_t u - \operatorname{div} \bigl( b(x,y) \nabla u \bigr) + V u + |u|^2 u = 0, 
    && (x,y) \in \Omega,\ t \in [0,T], \label{eq:nls-inhom-pde} \\
  & u(x,y,t) = 0, 
    && (x,y) \in \partial \Omega,\ t \in [0,T], \label{eq:bc-inhom} \\
  & u(x,y,0) = \frac{1}{10} \sin(\pi x) \sin(\pi y), 
    && (x,y) \in \Omega, \label{eq:ic1-inhom} \\
  & \partial_t u(x,y,0) = -\frac{i}{10} \sin(\pi x) \sin(\pi y), 
    && (x,y) \in \Omega, \label{eq:ic2-inhom}
\end{align}
\end{subequations}
where the coefficient is given by
\begin{equation}
  b(x,y) = \bigl[ (2.8 + x^2)(2.8 + y^2) \bigr]^2, \label{eq:b-inhom}
\end{equation}
and the smooth potential is the same as in \eqref{eq:V-hom}.

\begin{figure}[!hbt]
    \centering 
    \begin{minipage}[t]{0.48\textwidth} 
        \centering
        \includegraphics[width=\linewidth]{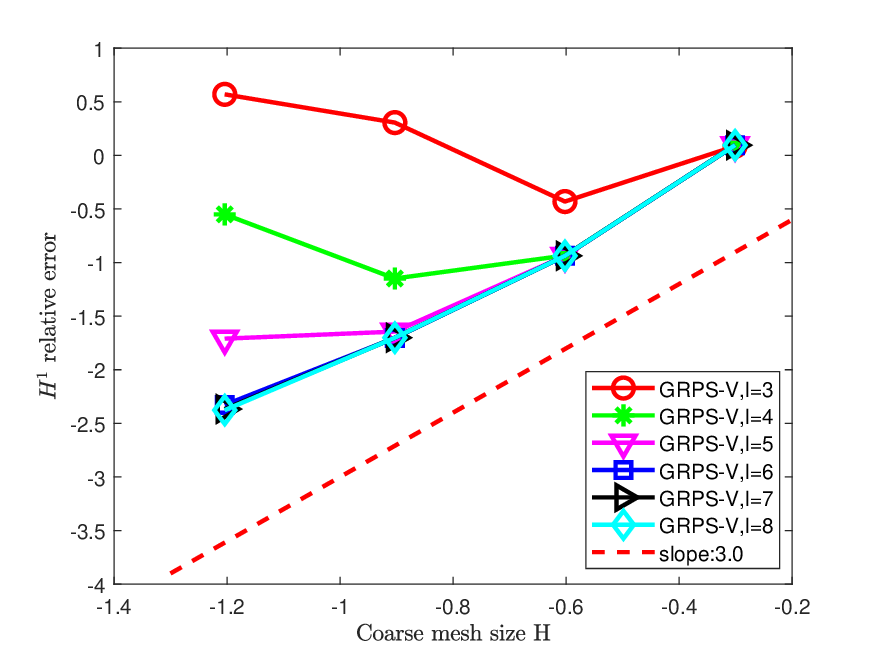} 
        \caption{ $H^1$ relative error.}
        \label{fig:example2_h1} 
    \end{minipage}%
    \hfill 
    \begin{minipage}[t]{0.48\textwidth}
        \centering
        \includegraphics[width=\linewidth]{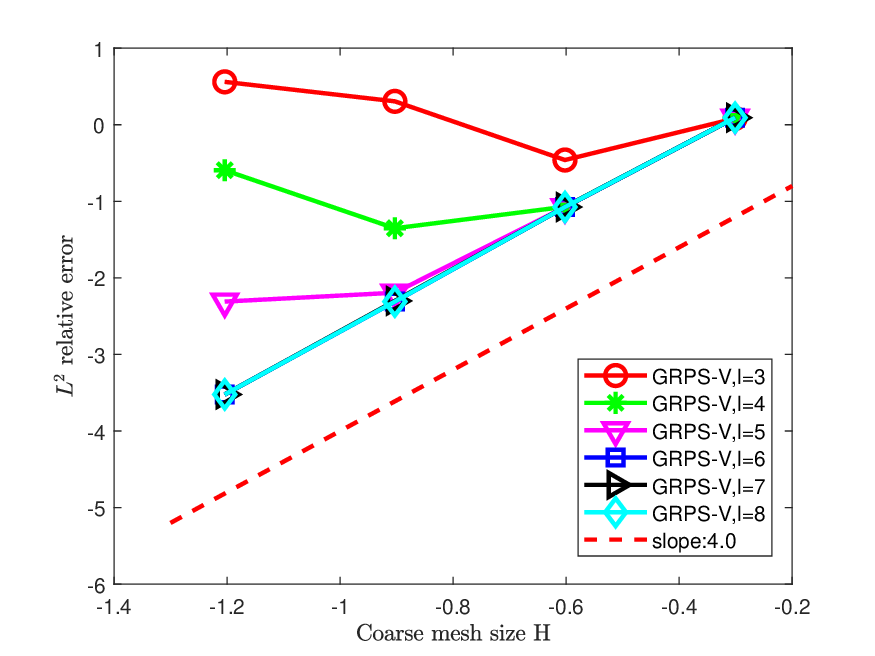}
        \caption{ $L^2$ relative error.}
        \label{fig:example2_l2}
    \end{minipage}
    \label{fig2} 
\end{figure}

Since the exact solution is unknown, we use a reference numerical solution computed by the finite element method on a fine mesh with $  h = 1/64  $ and time step $  \tau = 1.0 \times 10^{-2}  $. The convergence is then studied by varying the coarse mesh size $  H = 1/2, 1/4, 1/8, 1/16  $ while keeping the localization parameter $  l  $ in the range $3,\dots,8$. As shown in Figs.~\ref{fig:example2_h1}--\ref{fig:example2_l2}, the proposed method exhibits approximately fourth-order convergence in the $  L^2  $ norm and third-order convergence in the $  H^1  $ norm.  The numerical results thereby confirm the perfect approximation property of the LOD method, which aligns perfectly with the assertion of Theorem~\ref{main_thm_2.1}.
\end{example}



In the following, we present two benchmark examples (Examples~\ref{Example3} and~\ref{Example4}). The coefficient $  b(x,y)  $ is identical to that used in Examples~\ref{Example1} and~\ref{Example2}, respectively. The examples differ in the potential $  V(x,y)  $: Example~\ref{Example3} features a two-scale harmonic potential with regional optical lattice modulation, while Example~\ref{Example4} consists of a harmonic potential with a constant offset in one half-plane.

\begin{example}\label{Example3}
We consider the two-dimensional NLS equation with wave operator and constant coefficient $b(x,y) \equiv 1$. The potential has a two-scale structure and decomposes as $V = V_1 + V_2$, where $V_1$ is a smooth harmonic background
\[
V_1(x,y) = |x - 0.5|^2 + |y - 0.5|^2,
\]
and $V_2$ is a rapidly oscillating piecewise periodic perturbation:
\begin{equation}
V_2(x,y) =
\begin{cases}
  \bigl(0.01 + \cos(2\pi x / e_1)\bigr) \bigl(0.01 + \cos(2\pi y / e_1)\bigr)
    & 0 \leq x,y \leq 1/2, \\[6pt]
  \bigl(0.01 + \cos(2\pi x / e_2)\bigr) \bigl(0.01 + \cos(2\pi y / e_2)\bigr)
    & \text{otherwise},
\end{cases}
\label{eq:V2-ex3}
\end{equation}
with $e_1 = 1/8$ and $e_2 = 1/16$.

The governing equation is the same as in Example~\ref{Example1}, except for the modified initial conditions:
\begin{subequations}\label{eq:ic-ex3}
\begin{align}
  &u(x,y,0) = \frac{2}{5} \sin(\pi x) \sin(\pi y),
    && (x,y) \in \Omega, \label{eq:ic1-ex3} \\
  &\partial_t u(x,y,0) = -\frac{4\pi i}{5} \sin(\pi x) \sin(\pi y),
    && (x,y) \in \Omega. \label{eq:ic2-ex3}
\end{align}
\end{subequations}

\begin{figure}[!hbt]
    \centering 
    \begin{minipage}[t]{0.48\textwidth} 
        \centering 
        \includegraphics[width=\linewidth]{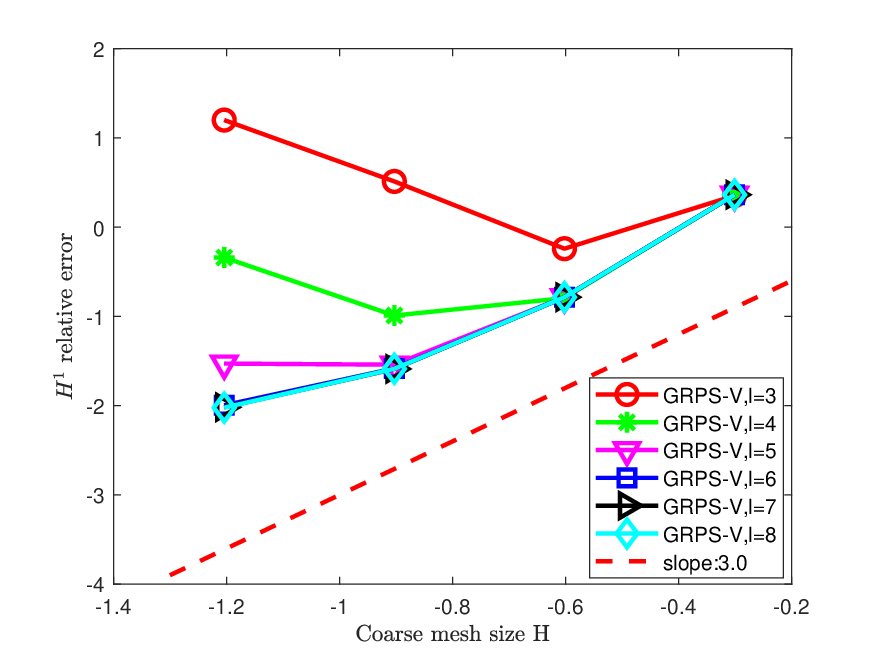} 
        \caption{Example 3: $H^1$ relative error.}
        \label{fig3_h1} 
    \end{minipage}%
    \hfill 
    \begin{minipage}[t]{0.48\textwidth}
        \centering
        \includegraphics[width=\linewidth]{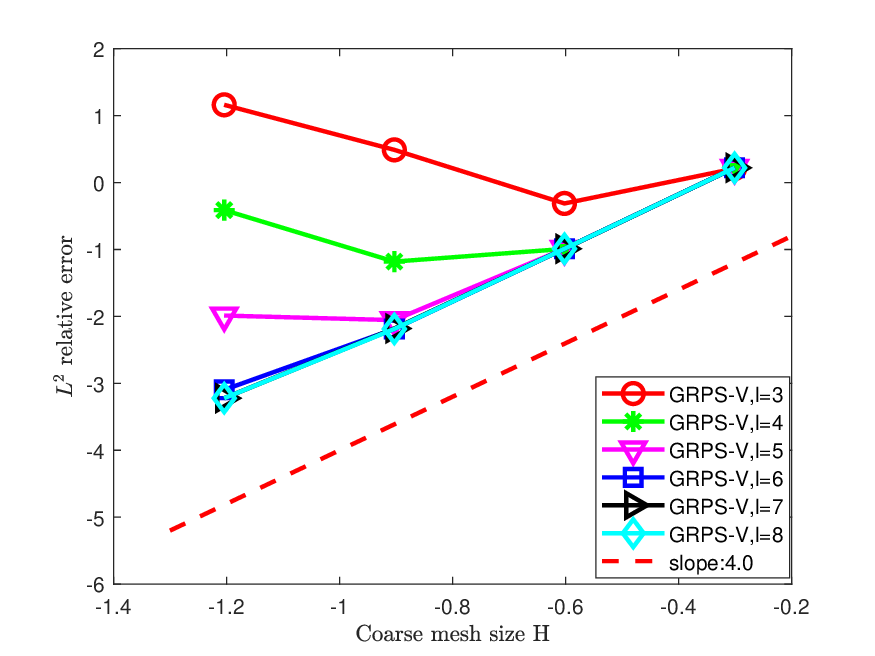}
        \caption{Example 3: $L^2$ relative error.}
        \label{fig3_l2}
    \end{minipage}
    \label{fig3} 
\end{figure}

Similarly, as no analytical solution is available, a reference solution is computed using the finite element method on a fine mesh with $  h = 1/64  $ and time step $  \tau = 1.0 \times 10^{-2}  $. The convergence behavior is studied by successively refining the coarse mesh size $  H = 1/2, 1/4, 1/8, 1/16  $ for localization parameters $  l = 3, \dots, 8  $. 
Figs.~\ref{fig3_h1} and~\ref{fig3_l2} demonstrate that the method attains approximately fourth-order convergence in the $  L^2  $ norm and third-order convergence in the $  H^1  $ norm.
\end{example}

\begin{example}\label{Example4}
In this numerical example, we consider the same model as in Example~\ref{Example2}, except for the piecewise shifted harmonic potential
\[
V(x,y) =
\begin{cases}
  \frac{1}{10}(x^2 + y^2) + 1.8 & x \geq 0, \\
  \frac{1}{10}(x^2 + y^2)        & \text{otherwise}.
\end{cases}
\]

\begin{figure}[!hbt]
    \centering 
    \begin{minipage}[t]{0.48\textwidth} 
        \centering
        \includegraphics[width=\linewidth]{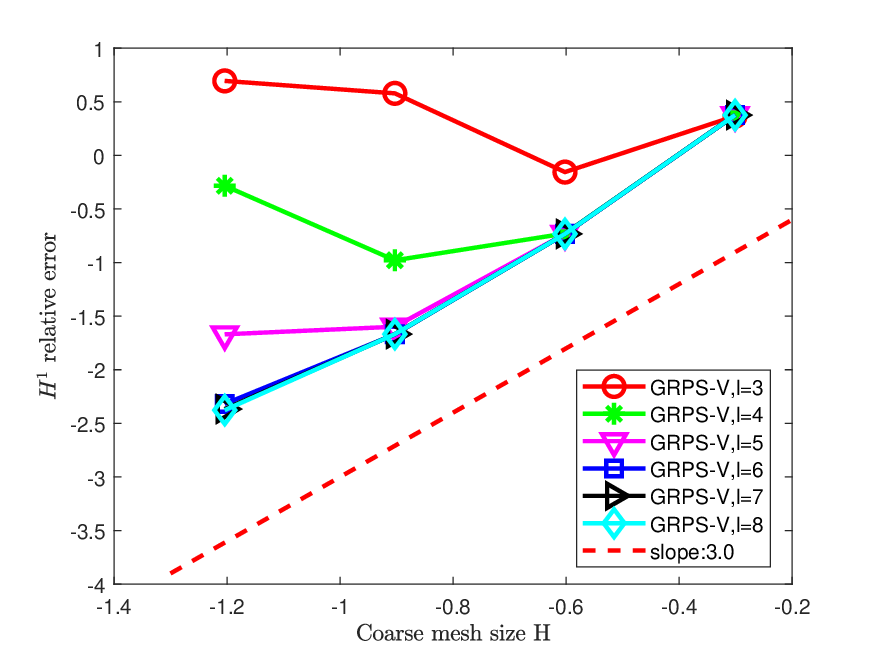} 
        \caption{Example 4: $H^1$ relative error.}
        \label{fig4_h1}
    \end{minipage}%
    \hfill 
    \begin{minipage}[t]{0.48\textwidth}
        \centering
        \includegraphics[width=\linewidth]{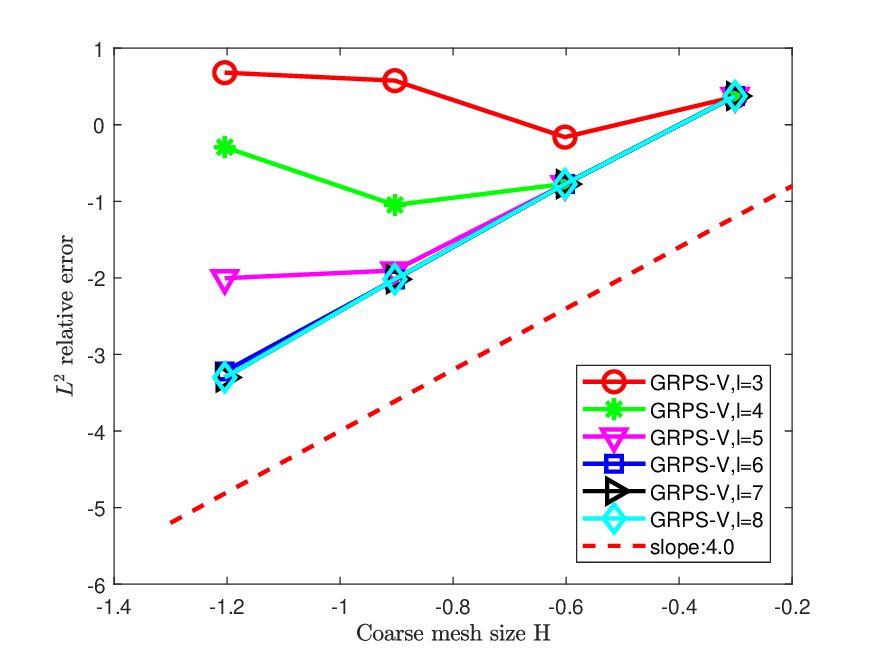}
        \caption{Example 4: $L^2$ relative error.}
        \label{fig4_l2}
    \end{minipage}
    \label{fig4} 
\end{figure}

With the same discretization parameters as in Example~\ref{Example3}, the numerical results demonstrate that the method continues to exhibit roughly fourth-order convergence in the $  L^2  $ norm and third-order convergence in the $  H^1  $ norm (see Figs.~\ref{fig4_h1}--\ref{fig4_l2}).
\end{example}

Finally, we conclude with a particularly challenging test case that combines a deterministically oscillatory coefficient with a randomly heterogeneous potential---a configuration designed to stress-test multiscale numerical methods.
\begin{example}\label{Example5}
In this numerical example, we consider the same model as in Example~\ref{Example2}, but with a random checkerboard-type multiscale coefficient $V(x,y)$ and a rough coefficient $b(x,y)$. The coefficient $b(x,y)$ is given by
\begin{equation}
\begin{split}
b(x,y) &= \frac{1}{6} \Biggl[ \frac{3 + \sin(2\pi x / e_1)}{3 + \sin(2\pi y / e_1)}
  + \frac{3 + \sin(2\pi y / e_2)}{3 + \cos(2\pi x / e_2)} \\
  &\quad + \frac{3 + \cos(2\pi x / e_3)}{3 + \sin(2\pi y / e_3)}
  + \frac{3 + \sin(2\pi x / e_4)}{3 + \cos(2\pi y / e_4)} \\
  &\quad + \frac{3 + \cos(2\pi x / e_5)}{3 + \sin(2\pi y / e_5)}
  + \sin(4x^2 y^2) + 1 \Biggr],
\end{split}
\label{eq:b-multiscale}
\end{equation}
where the frequency parameters are $e_1 = 1/5$, $e_2 = 1/13$, $e_3 = 1/17$, $e_4 = 1/31$, and $e_5 = 1/65$.

The random checkerboard-type multiscale coefficient $V(x,y)$ is defined on the domain $\Omega = [0,1]^2$. It is generated on a fine grid with mesh size $h = 1/128$, where each cell independently takes the value $0.05$ or $20$ with equal probability $1/2$. As illustrated in Fig.~\ref{fig5_left}, the resulting coefficient field is highly oscillatory and exhibits no scale separation.

For the numerical tests, a reference solution is computed on a fine mesh with $  h = 1/128  $, while approximate solutions are obtained on coarse meshes with $  N_c = 2, 4, 8, 16  $. Fig.~\ref{fig5_right} shows the relative $  L^2  $ error $  \|u^N - u_H^{l,N}\| / \|u^N\|  $ versus the coarse mesh size $  H  $ on a log-log scale, exhibiting approximately second-order convergence. Due to the multiscale nature of both coefficients $  b(x,y)  $ and $  V(x,y)  $, the proposed method struggles to fully capture this feature. After extensive numerical testing, Fig.~\ref{fig5_right} shows the most favorable results. This challenging case leads to a reduced convergence rate in the relative $  L^2  $ error.

\begin{figure}[!hbt]
    \centering 
    \begin{minipage}[t]{0.48\textwidth} 
        \centering
        \includegraphics[width=\linewidth]{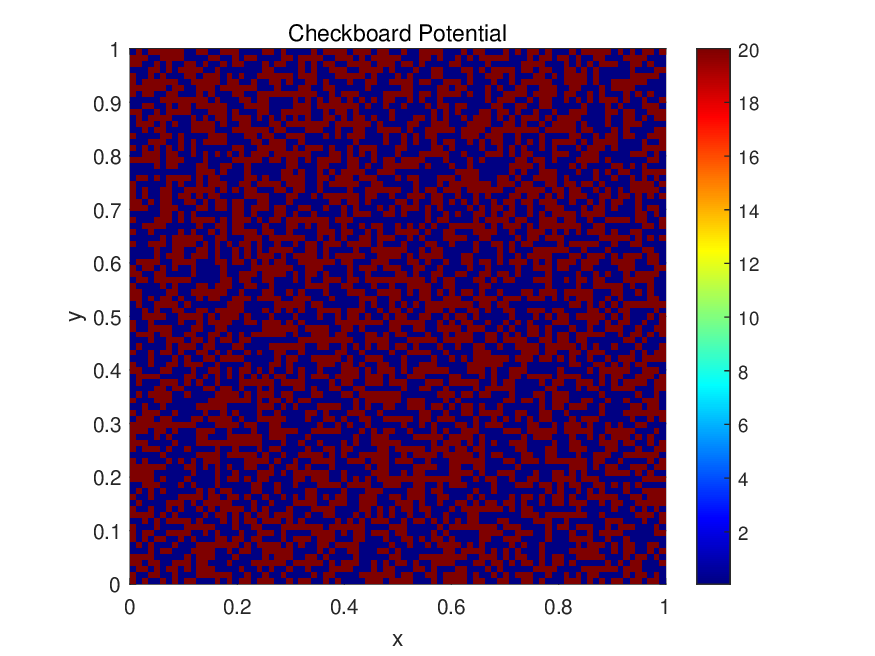} 
        \caption{Checkboard Potential.}
        \label{fig5_left}
    \end{minipage}%
    \hfill
    \begin{minipage}[t]{0.48\textwidth}
        \centering
        \includegraphics[width=\linewidth]{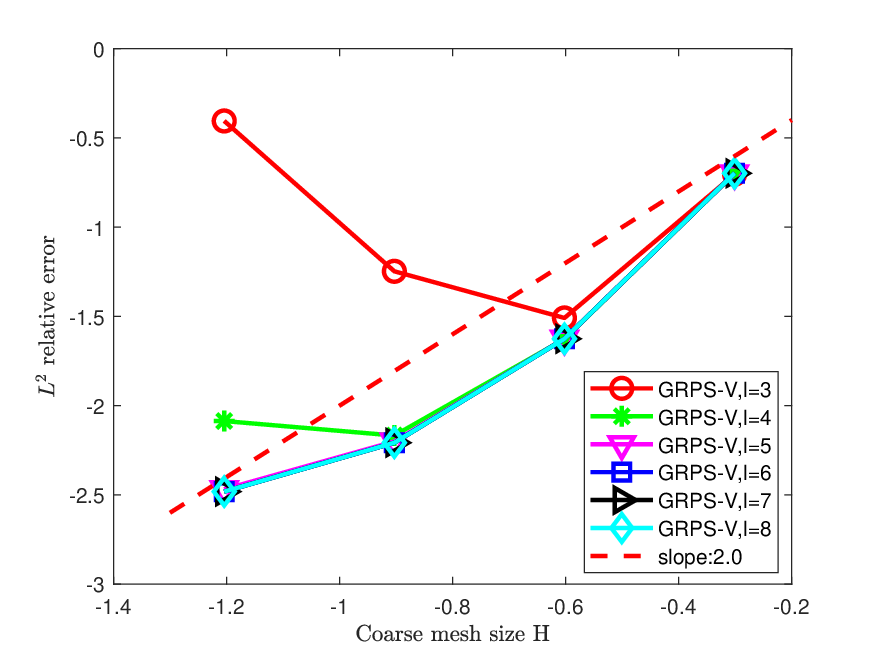}
        \caption{$L^2$ relative error for Checkboard Potential.}
        \label{fig5_right}
    \end{minipage}
    \label{fig5}
\end{figure}    
\end{example}

\section{Proof of Main Theorems}
\label{sec:Proof_of_Main_Theorems}
In this section, we prove the main theorems of the paper: Theorems~\ref{main_thm_2.1}, \ref{thm:conservation_property}, and \ref{thm2.2}.
First, Theorem~\ref{thm:conservation_property} establishes the energy conservation property of the proposed numerical method. Second, Theorem~\ref{thm2.2} proves the boundedness of the numerical solution obtained by the LOD approach. Finally, building upon Theorem~\ref{thm2.2} and Lemmas~\ref{lem3.1}--\ref{lem3.5}, Theorem~\ref{main_thm_2.1} presents the central result of this work — namely, the optimal convergence order of the proposed method.

\subsection{Proof of Theorem \ref{thm:conservation_property}}
First, we prove the discrete energy conservation property of the proposed scheme. The proof is based on taking the real part of the discrete equation.
\begin{proof}
Taking the inner product of \eqref{uh} with \( \delta_{\hat{t}} u^n_H \) and extracting the real part, we obtain:
\begin{multline*} \label{long-eq}
\Re \Big[ (\delta^2_t u^n_H, \delta_{\hat{t}} u^n_H) + i(\delta_{\hat{t}} u^n_H, \delta_{\hat{t}} u^n_H) + (\nabla \bar{u}^n_H, \nabla \delta_{\hat{t}} u^n_H) \\
+ (V \bar{u}^n_H, \delta_{\hat{t}} u^n_H) + \big(\tilde{f}(|u^{n+1}_H|^2, |u^{n-1}_H|^2) \bar{u}^n_H, \delta_{\hat{t}} u^n_H\big) \Big] = 0.
\end{multline*}
This identity can be rewritten as:
\[
M_1 + M_2 + M_3 + M_4 + M_5 = 0,
\]
where
\begin{align*}
M_1 &= \Re(\delta^2_t u^n_H, \delta_{\hat{t}} u^n_H) = \frac{1}{2} \delta_{\bar{t}} (\|\delta_t u^n_H\|^2), \\
M_2 &= \Re(i \delta_{\hat{t}} u^n_H, \delta_{\hat{t}} u^n_H) = 0, \\
M_3 &= \Re(\nabla \bar{u}^n_H, \nabla \delta_{\hat{t}} u^n_H) = \frac{1}{4\tau} \left(|u^{n+1}_H|^2_1 - |u^{n-1}_H|^2_1\right), \\
M_4 &= \Re(V \bar{u}^n_H, \delta_{\hat{t}} u^n_H) = \frac{1}{4\tau} \int_{\Omega} V (|u^{n+1}_H|^2 -|u^{n-1}_H|^2 ) dx, \\
M_5 &= \Re(\tilde{f}(|u^{n+1}_H|^2, |u^{n-1}_H|^2) \bar{u}^n_H, \delta_{\hat{t}} u^n_H) = \frac{1}{4\tau} \left(F(|u^{n+1}_H|^2) - F(|u^{n-1}_H|^2)\right).
\end{align*}

Then, multiplying both sides by \( 4\tau \), we obtain:
\[
2\tau \delta_{\bar{t}} (\|\delta_t u^n_H\|^2) + \left(|u^{n+1}_H|^2_1 - |u^{n-1}_H|^2_1\right) + \int_{\Omega} V (|u^{n+1}_H|^2 -|u^{n-1}_H|^2 )dx   + \left(F(|u^{n+1}_H|^2) - F(|u^{n-1}_H|^2)\right) = 0.
\]
By exploiting the telescoping property of the summation, we derive the following discrete energy conservation law:
\[
\begin{aligned}
\Bigl[ \|\delta_t u^n_H\|^2 
&+ \frac{1}{2}\bigl( |u^{n+1}_H|_1^2 + |u^n_H|_1^2 \bigr) 
+ \frac{1}{2}\int_{\Omega} V \bigl( |u^{n+1}_H|^2 + |u^n_H|^2 \bigr) \,\mathrm{d}x \\
&+ \frac{1}{2}\bigl( F(|u^n_H|^2) + F(|u^{n-1}_H|^2) \bigr) \Bigr] = E^{n-1}.
\end{aligned}
\]
Consequently, the scheme is energy-preserving, satisfying \( E^n = E^{n-1} = \cdots = E^0 \), where the initial energy is given by:
\[
E^0 = \|\delta_t u^0_H\|^2 + \frac{1}{2}(|u^1_H|^2_1 + |u^0_H|^2_1) + \frac{1}{2}\int_{\Omega} V (|u^{1}_H|^2 +|u^{0}_H|^2 )dx + \frac{1}{2}\left[F(|u^1_H|^2) + F(|u^0_H|^2)\right].
\]
\end{proof}

Next, we state a technical lemma that will be used in the proof of Theorem~\ref{thm2.2} and Lemma~\ref{lem3.5}.
\begin{lemma}\label{lem2.5}
\cite{Lix} For any discrete function \( \{\phi^n\} \), the following inequality holds:
\[
\|\phi^{n+1}\|^2 - \|\phi^{n-1}\|^2 \leq 2\tau \left[\|\delta_{\hat{t}} \phi^n\|^2 + \frac{1}{2}(\|\phi^{n+1}\|^2 + \|\phi^{n-1}\|^2)\right].
\]
\end{lemma}

Finally, introduce a lemma that will be used in the proof of Theorem~\ref{thm2.2}.
Since all norms on $(V_{\mathrm{LOD}})^N$ are equivalent for fixed $N$ and $H$, the map $A$ is continuous and compact in this finite-dimensional space.
\begin{lemma}
[Schaefer's fixed point theorem {\cite[Chapter 9.2, Theorem 4]{Evans-PDE}}]
\label{THMSchaefer}
If $A:(V_{\mathrm{LOD}})^N\mapsto (V_{\mathrm{LOD}})^N$ is a continuous and compact mapping, 
and the set
	\begin{equation}\label{def-B}
	\mathbb{B}=\{w \in (V_{\mathrm{LOD}})^N: \exists\,\theta\in[0,1] \,\,\,\mbox{such that}\,\,\, w=\theta Aw \}	
	\end{equation}
	is bounded in $(V_{\mathrm{LOD}})^N$, then the map $A$ has at least one fixed point.
\end{lemma}

\subsection{Proof of Theorem \ref{thm2.2}}

We now establish the existence, uniqueness, and boundedness of the solution $u^n_H$ in $L^{\infty}$-norm.
\begin{proof}
Define \(\tilde{u}_{H}^{n} = R_H u^n\).  
Then for all \( v_H \in V_{\mathrm{LOD}} \), \(\tilde{u}_{H}^{n}\) satisfies 
\begin{align}\label{exact-CN-LM}
& \bigl( \delta^{2}_{t} \tilde{u}_{H}^{n}, v_H \bigr) 
   + \bigl( i \delta_{\hat{t}} \tilde{u}_{H}^{n}, v_H \bigr) 
   + \bigl( b \nabla \tilde{\bar{u}}_{H}^{n}, \nabla v_H \bigr) 
   + \bigl( V \tilde{\bar{u}}_{H}^{n}, v_H \bigr) \nonumber \\
& \quad + \Bigl( \tilde{f}\bigl(|\tilde{u}_{H}^{n+1}|^{2}, |\tilde{u}_{H}^{n-1}|^{2}\bigr) \,
            \tilde{\bar{u}}_{H}^{n}, v_H \Bigr)
   = \bigl( R_*^{n}, v_H \bigr),
\end{align}
where \( R_*^{n} \in V_{\mathrm{LOD}} \) denotes the \emph{consistency error} introduced by the finite element method.
Thus, the consistency error \( R_*^n \) can be decomposed as follows:
\begin{align}\label{def-d-start-n}
\begin{split}
(R_*^n, v_H) =&\ 
\bigl((R_H - P_H)\delta_t^2 u^n, v_H\bigr) 
+ \bigl(\delta_t^2 u^n - \partial_t^2 u^n, v_H\bigr) \\[2pt]
&+ i \bigl((R_H - P_H)\delta_{\hat{t}} u^n, v_H\bigr) 
+ i\bigl(\delta_{\hat{t}} u^n - \partial_t u^n, v_H\bigr) \\[2pt]
&+ \bigl(V(\tilde{\bar{u}}_H^n - \bar{u}^n), v_H\bigr) 
+ \bigl(b\nabla(\tilde{\bar{u}}_H^n - \bar{u}^n), \nabla v_H\bigr) \\[2pt]
&+ \Bigl(
    \tilde{f}\bigl(|\tilde{u}_H^{n+1}|^2, |\tilde{u}_H^{n-1}|^2\bigr)\tilde{\bar{u}}_H^n
    - \tilde{f}\bigl(|u^{n+1}|^2, |u^{n-1}|^2\bigr)\bar{u}^n,
    v_H
\Bigr).
\end{split}
\end{align}
Subtracting \eqref{exact-CN-LM} from \eqref{uh} and defining the error as \( e_H^n = u_H^n - \tilde{u}_H^n \), we obtain the following error equation:
\begin{align}\label{error-CN-LM}
\begin{split}
& \bigl( \delta^2_{t} e_H^n, v_H \bigr) 
   +  i \bigl(\delta_{\hat{t}} e_H^n, v_H \bigr) 
   + \bigl( b\nabla \bar{e}_H^n, \nabla v_H \bigr) 
   + \bigl( V\bar{e}_H^n, v_H \bigr) \\[2pt]
&\quad + \Bigl( 
    \tilde{f}\bigl(|u_H^{n+1}|^2, |u_H^{n-1}|^2\bigr) \bar{u}_H^n
    - \tilde{f}\bigl(|\tilde{u}_H^{n+1}|^2, |\tilde{u}_H^{n-1}|^2\bigr) \tilde{\bar{u}}_H^n,
    v_H
\Bigr) \\[2pt]
&= -\bigl( R_*^n, v_H \bigr), \qquad \forall v_H \in V_{\mathrm{LOD}}.
\end{split}
\end{align}

Now, we will prove that $u_H^n$ is a solution of \eqref{uh} if and only if $e_{H}^{n}$ is a solution of \eqref{error-CN-LM} with $u_H^n=\tilde{u}_H^n+e_H^n$.

To show that solutions of \eqref{error-CN-LM} exist for sufficiently small $\tau$ and $H$, we introduce a map $A$ from  $w_H^n$ to $e_H^n$ as follows:
For  $ \forall \ w=(w_H^n)_{n=1}^N \in (V_{\mathrm{LOD}})^N$, we define
\begin{align}\label{def-phi-w}
\phi[w]
= \min\bigg(\frac{1}{\displaystyle\max_{1\le n\le N}\|w_h^n\|_{L^\infty}},1\bigg),
\end{align}
and $u_H^n=\tilde{u}_{H}^{n}+\phi[w] w_H^n $.
By construction, this definition ensures the following point-wise bound
\begin{align}\label{property-phi-w}
\|u_H^n-\tilde{u}_{H}^{n}\|_{L^\infty}=\|\phi[w] w_H^n\|_{L^\infty}\le 1,\ \mbox{for}\,\,\, n=1,\cdots,N.
\end{align}

{\bf{\it Part I: Boundedness of the set $\mathbb{B}$.}}\\

For $w = (w_H^n)_{n=1}^N \in \mathbb{B}$, we have $w = \theta A w$. Setting $e = A w = (e_H^n)_{n=1}^N$, it follows that $w = \theta e$, where $e_H^n$ satisfies \eqref{error-CN-LM} with $u_H^n = \tilde{u}_H^n + \phi[\theta e] \theta e_H^n$. Consequently,
\begin{align}\label{3.21h}
\max_{1\le n\le N} \|u_H^n - \tilde{u}_H^n\|_{L^\infty}
&\le \phi[\theta e] \max_{1\le n\le N} \|\theta e_H^n\|_{L^\infty} \le 1,
\end{align}
where the last inequality follows from \eqref{def-phi-w}.
Thus, $u = (u_H^n)_{n=1}^N$ lies in an $L^\infty$ neighborhood of the reference solution $\tilde{u} = (\tilde{u}_H^n)_{n=1}^N$ in $(V_{\mathrm{LOD}})^N$. Applying the elliptic projection \eqref{uproj} together with the inverse inequality immediately yields
\begin{equation}\label{3.21h2}
\|\tilde{u}_H^n - u^n\|_{L^\infty} \le C H \|u^n\|_{H^2}.
\end{equation}
Combining \eqref{3.21h} and \eqref{3.21h2}, we immediately obtain the uniform $  L^\infty  $ bound
$$\max_{1\le n\le N} \|u_H^n\|_{L^\infty} \le C.$$

To establish the boundedness of $e$ in \((V_{\mathrm{LOD}})^N\), 
we proceed as follows. Taking the inner product of \eqref{error-CN-LM} with \(\delta_{\hat{t}} e_H^n\), 
extracting the real part, and exploiting the uniform boundedness of \(\{u_H^n\}\) yields
\begin{align}\label{3.24h}
\begin{split}
\frac{\|\delta_{t} e_H^{n}\|^2 - \|\delta_{t} e_H^{n-1}\|^2}{2\tau} 
&+ \frac{b_{*}}{4\tau}\bigl(\|e_H^{n+1}\|_1^2 - \|e_H^{n-1}\|_1^2\bigr) \\
&\leq C_1\Bigl(
    \|e_H^{n+1}\|^2 + \|e_H^{n-1}\|^2 
    + \|R_*^n\|^2 + \|\delta_{\hat{t}} e_H^n\|^2
\Bigr).
\end{split}
\end{align}
The estimate for the nonlinear term follows by an argument analogous to that used for the term $  I_3  $ in Lemma~\ref{lem3.5}.

By applying Lemma \ref{lem2.5}, \eqref{3.24h}, we obtain the rewritten form
\begin{align}
\begin{split}\label{3.25h}
 &\frac{\|\delta_{t}e_{H}^{n}\|^2-\|\delta_{t}e_{H}^{n-1}\|^2}{2\tau}+\frac{\|e_{H}^{n+1}\|^2-\|e_{H}^{n-1}\|^2}{2\tau}+\frac{b_{*}}{4\tau}(\|e^{n+1}_H\|^2_1-\|e^{n-1}_H\|^2_1)\\
&\leq C_1(\|e^{n+1}_H\|^2+\|e^{n-1}_H\|^2+\|\delta_{t}e^{n}_H\|^2+\|\delta_{t}e^{n-1}_H\|^2+\|R_*^n\|^2).
\end{split}
\end{align}
Multiplying both sides of the above inequality by $2\tau$ and summing over $n$, while taking $R_H u_0 = u_0$, yields
\begin{align}
\begin{split}
  &\|\delta_t e_H^n\|_2^2 + \frac{b_*}{2} \bigl( \|e_H^{n+1}\|_1^2 + \|e_H^n\|_1^2 \bigr) + \bigl( \|e_H^{n+1}\|^2 + \|e_H^n\|^2 \bigr) \\
  &\leq C(\tau^4 + H^4) + C\tau \sum_{k=1}^n \Bigl[ \|\delta_t e_H^k\|_2^2 + \bigl( \|e_H^{k+1}\|_1^2 + \|e_H^k\|_1^2 \bigr) + \bigl( \|e_H^{k+1}\|^2 + \|e_H^k\|^2 \bigr) \Bigr],
\end{split}
\end{align}
where  $\|R_*^n\| \leq C(\tau^2 + H^2)$ follows readily from \eqref{def-d-start-n}.

Applying Gronwall's inequality and assuming the time step size $  \tau  $ is sufficiently small, we obtain 
\begin{align}\label{e_hn-L2-error}
\max_{1\le n\le N} \bigl( |\delta_t e_H^n| + |e_H^n|_1 + |e_H^n| \bigr) \le C(\tau^2 + H^2).
\end{align}
For $\tau \leq H$, applying the inverse inequality to \eqref{e_hn-L2-error} yields
\begin{align*}
  \|e_H^n\|_{L^\infty} 
  \leq C H^{-d/2} (\tau^2 + H^2) 
  \leq C H^{2 - d/2}.
\end{align*}
When $\tau \geq H$, \eqref{e_hn-L2-error} implies
\begin{equation*}
  \max_{1 \leq n \leq N} \bigl( \|\delta_t e_H^n\| + \|e_H^n\| \bigr) \leq C \tau^2.
\end{equation*}
Consequently, we obtain the estimates
\begin{align*}
  \|\delta_{\hat{t}} e_H^n\| 
  &\leq \frac{\|\delta_t e_H^n\| + \|\delta_t e_H^{n-1}\|}{2} \leq C \tau^2, \\
  \|\delta_t^2 e_H^n\| 
  &\leq \frac{\|\delta_t e_H^n\| + \|\delta_t e_H^{n-1}\|}{\tau} \leq C \tau.
\end{align*}
Applying \eqref{error-CN-LM}, we obtain the following $H^2$ estimate:
\begin{align}\label{H2h-ehn-2}
  b_* \|\Delta \bar{e}_H^n\| 
  \leq C \bigl( \|\delta_t^2 e_H^n\| + \|\delta_{\hat{t}} e_H^n\| 
                + \|e_H^{n+1}\| + \|e_H^{n-1}\| + \|R_*^n\| \bigr)
  \leq C (\tau + \tau^2 + \|R_*^n\|) \leq C \tau.
\end{align}
By applying the discrete interpolation inequality from Lemma 3.4 of \cite{hu2022optimal} together with \eqref{e_hn-L2-error}--\eqref{H2h-ehn-2}, we obtain the following $  L^\infty  $ error estimate:
$$\|e_H^n\|_{L^\infty} \le C \|e_H^n\|^{1 - \frac{d}{4}} \|\Delta e_H^n\|^{\frac{d}{4}} \le C \tau^{2 - \frac{d}{4}}, \quad \text{for } \tau \geq H.$$

Finally, by combining the two cases $  \tau \leq H  $ and $  \tau \geq H  $, we obtain the following $  L^\infty  $ error estimate:
$$  \|e_H^n\|_{L^\infty} \le C \Bigl( \tau^{2 - d/4} + H^{2 - d/2} \Bigr).$$

By combining the two cases where $\tau\le H$ and $\tau\ge H$, we obtain
\begin{equation}\label{2.26}
 \|e_H^n\|_{L^\infty}
\le C(\tau^{2-\frac{d}{4}}+H^{2-\frac{d}{2}}).
\end{equation}
This shows that the set $  \mathbb{B}  $ defined in \eqref{def-B} is bounded in $  (V_{\mathrm{LOD}})^N  $ with respect to the $  L^\infty  $ norm. By applying Lemma~\ref{THMSchaefer}, the operator $  A  $ admits a fixed point, denoted $  e = (e_H^n)_{n=1}^N  $.

Furthermore, when $\tau$ and $H$ are small enough, the inequality \eqref{2.26} leads to
\begin{align}\label{Linfty-ehn-0}
\|e_H^n\|_{L^\infty}\le 1/2.
\end{align}

{\bf{\it Part II: Existence, Uniqueness and $L^{\infty}$-norm boundedness of finite element solution $u^n_h$.}}\\
By the definition of $\phi[\theta e]$ in \eqref{def-phi-w}, the bound \eqref{Linfty-ehn-0} implies $\phi[\theta e] = 1$. Consequently, the fixed point $e = (e_H^n)_{n=1}^N$ satisfies \eqref{error-CN-LM} with $u_H^n = \tilde{u}_H^n + e_H^n$. In fact, $u_H^n$ solves \eqref{uh} if and only if $e_H^n$ solves \eqref{error-CN-LM} with $u_H^n = \tilde{u}_H^n + e_H^n$. This establishes the \emph{existence} of a numerical solution to the implicit scheme \eqref{uh}.

Furthermore, combining \eqref{3.21h2} and \eqref{Linfty-ehn-0} yields the uniform $L^\infty$ bound
\begin{equation*}
  \max_{1\leq n \leq N} \|u_H^n\|_{L^\infty} \leq C.
\end{equation*}

The uniqueness of the solution can be established in a manner analogous to the argument presented in \cite{hu2022optimal}.
\end{proof}

\subsection{ Optimal \texorpdfstring{$L^p$}{Lp} Error Estimate of the LOD Solution }
\label{sec:Optimal_Lp_LOD_estimate}
In order to make an analysis of the LOD solution, we need analyze the
$L^p$ error estimate of the LOD solution. First, we establish a useful lemma by analyzing an elliptic problem \eqref{uproj}.
 \begin{lemma}\label{lem3.1}
Under Assumption~\ref{ass:regularity}, for the elliptic projection $R_H u^n$ of $u^n$ at time $t^n$ ($1 \leq n \leq N$) defined in \eqref{uproj}, we have the following estimates for $2 \leq p < \infty$:
\begin{align}\label{Phu}
\begin{split}
  \max_{1\leq n \leq N} \|R_H u^n - u^n\|_{L^p} &\leq C \|u\|_{W^{4,p}} H^4, \\
  \max_{1\leq n \leq N} \|\partial_t (R_H u - u)\|_{L^p} &\leq C \|\partial_t u\|_{W^{4,p}} H^4, \\
  \max_{1\leq n \leq N} \|\partial_t^2 (R_H u - u)\|_{L^p} &\leq C \|\partial_t^2 u\|_{W^{4,p}} H^4.
\end{split}
\end{align}

\begin{proof}
 To prove \eqref{Phu}, we use a duality argument by considering the auxiliary problem
\begin{align}
\begin{split}
 L\omega=-{\rm div}(b(\bm{x})\nabla\omega)&={\rm sgn}(R_H u^n- u^n)|R_H u^n- u^n|^{p-1}, \ in \ \Omega,\\
\omega &= 0, \  \rm on \ \partial \Omega.
\end{split}
\end{align}
is uniquely solvable for $\omega \in L^p$ and has the following regularity
\begin{align}
\|\omega\|_{W^{2,q}}\leq C\|L\omega\|_{L^q}=C\|R_H u^n- u^n\|^{p-1}_{L^p}.\label{regularity}
\end{align}
Applying \eqref{uproj}, $\rm H\ddot{o}lder's$ inequality with $\frac{1}{p} + \frac{1}{q} = 1$, and the regularity \eqref{regularity}, we obtain
\begin{align*}
  \|R_H u^n - u^n\|_{L^p}^p 
  &= a(R_H u^n - u^n, \omega) \\
  &= a(R_H u^n - u^n, \omega - I_H \omega) \\
  &\leq C \|R_H u^n - u^n\|_{W^{1,p}} \|\omega - I_H \omega\|_{W^{1,q}} \\
  &\leq C H \|R_H u^n - u^n\|_{W^{1,p}} \|\omega\|_{W^{2,q}} \\
  &\leq C H \|R_H u^n - u^n\|_{W^{1,p}} \|R_H u^n - u^n\|_{L^p}^{p-1},
\end{align*}
where $I_H$ denotes the interpolation operator.

Using \eqref{l2.3}, we have
\begin{equation*}
a(R_H u - u, \varphi) = (g, \varphi) \quad \forall \varphi \in W \cap W^{1,q}(\Omega),
\end{equation*}
where $  g = \Delta u + V u  $. Besides, we have
\begin{align*}
  (g,\varphi) & =(g-P_Hg,\varphi)=(g-P_Hg,\varphi-P_H\varphi)\\
             &\leq H^2\|g\|_{W^{2,p}}H\|\varphi\|_{W^{1,q}}\\
             &\leq H^3\|g\|_{W^{2,p}}\|\varphi\|_{W^{1,q}},
\end{align*}
where we use $\rm H\ddot{o}lder's$ inequality and the error estimates of the $L^2$-projection.
\begin{align}
\begin{split}
  b_{*} \|R_Hu-u\|_{W^{1,p}}&\leq\sup_{0\neq\varphi\in W\cap{W^{1,q}(\Omega)}}\frac{a(R_Hu-u,\varphi)}{\|\varphi\|_{W^{1,q}}} =\sup_{0\neq\varphi\in W\cap{W^{1,q}(\Omega)}}\frac{(g,\varphi)}{\|\varphi\|_{W^{1,q}}}\\
   &\leq H^3\|g\|_{W^{2,p}}.
\end{split}
\end{align}
 Then, we have
\begin{align*}
\|R_Hu^n-u^n\|_{L^p}\leq CH\|R_Hu^n-u^n\|_{W^{1,p}}
\leq CH^{4}\|u\|_{W^{4,p}}.
\end{align*}
Following the same argument as above, we readily obtain
\begin{align*}
  \max_{1\leq n \leq N}\|\partial_t(R_Hu-u)\|_{L^p}&\leq C\|\partial_t u\|_{W^{4,p}}H^{4},\\
\max_{1\leq n \leq N}\|\partial^2_t(R_Hu-u)\|_{L^p}&\leq C\| \partial^2_t u\|_{W^{4,p}}H^{4}.
\end{align*}
\end{proof}
 \end{lemma}

The following result appears in Lemma 3.2 of \cite{hu2021} and will be used in the proof of Lemma~\ref{lem3.5}.
\begin{lemma}\label{lem3.12}
  Under Assumption \ref{ass:regularity}, the following estimates hold for the temporal discretization error:
  \begin{align}
   \|\partial^2_tu^n-\delta^2_tu^n\|&\leq C\tau^{\frac{3}{2}}\|\partial^4_tu\|_{L^2(t^{n-1},t^{n+1};L^2)}. \\
   \|\partial_t u^n-\delta_{\hat{t}}u^n\|&\leq C\tau^{\frac{3}{2}}\|\partial^3_tu\|_{L^2(t^{n-1},t^{n+1};L^2)}.
  \end{align}
 \end{lemma}

 Now, we will derive optimal $L^p$ error estimate between the elliptic projection and the Galerkin solution.
\begin{lemma}\label{lem3.5}
Suppose that $R_Hu^{n}$ is the elliptic projection at $t = t^n$ with $1 \leq n \leq N$ defined in \eqref{uproj}and $u^{n}_H$ is the Galerkin
solution satisfying the equation \eqref{uh}, then, for
$2 \leq p < \infty$, we have
\begin{align}
  \max_{1\leq n \leq N}\|R_Hu^n- u^n_H\|_{1}+\max_{1\leq n \leq N}\|R_Hu^n- u^n_H\|_{L^p}\leq C (H^{4}+\tau^2).
\end{align}
\begin{proof}
By adding equation \eqref{uproj} to equation \eqref{u} evaluated at $  t = t^n  $ and then subtracting equation \eqref{uh}, we obtain the following expression
\begin{align}
\begin{split}
 &(\partial^2_t u^{n}-\delta^2_{t}u^{n}_{H},\varphi)+i(\partial_t u^{n}-\delta_{\hat{t}}u^{n}_{H},\varphi)+(b\nabla (R_H\bar{u}^{n}-\bar{u}^{n}_{H}),\nabla \varphi)\\
 &=-(V(R_H\bar{u}^{n}-\bar{u}^{n}),\varphi)-(f(|{u}^{n}|^2){u}^{n}-\tilde{f}(|{u}^{n+1}_H|^2,|{u}^{n-1}_H|^2)\bar{u}^{n}_H,\varphi)+(r^n,\varphi),\label{3.25}
\end{split}
\end{align}
where we use $u^n=\frac{u^{n+1}+u^{n}}{2}+O(\tau^2)=\bar{u}^n+O(\tau^2)=\bar{u}^n+r^n$.\\

Next, we estimate the terms in \eqref{3.25}. Defining $\rho^n = R_H u^{n} - u^{n}H$ and $\eta^n = u^{n} - R_H u^{n}$, and choosing $\varphi = \delta{\hat{t}}\rho^{n}$ in \eqref{3.25}, we rewrite it as
\begin{align}
\begin{split}
&(\delta^2_{t}\rho^{n},\delta_{\hat{t}}\rho^{n})+i(\delta_{\hat{t}}\rho^{n},\delta_{\hat{t}}\rho^{n})+(b\nabla \bar{\rho}^{n},\nabla \delta_{\hat{t}}\rho^{n})\\
&=-(\partial^2_t u^{n}-\delta^2_{t}u^{n},\delta_{\hat{t}}\rho^{n})-(\delta^2_{t}\eta^{n},\delta_{\hat{t}}\rho^{n})-i(\partial_t u^{n}-\delta_{\hat{t}}u^{n},\delta_{\hat{t}}\rho^{n})-i(\delta_{\hat{t}}\eta^{n},\delta_{\hat{t}}\rho^{n})\\
&+(V\bar{\eta}^{n},\delta_{\hat{t}}\rho^{n})+(r^n,\delta_{\hat{t}}\rho^{n})-(f(|{u}^{n}|^2){u}^{n}-\tilde{f}(|{u}^{n+1}_H|^2,|{u}^{n-1}_H|^2)\bar{u}^{n}_H,\delta_{\hat{t}}\rho^{n}).
\end{split}
\label{eq:total_identity}
\end{align}
Taking the real part of \eqref{eq:total_identity}, we obtain
\begin{equation}\label{3.10}
  \frac{\|\delta_{t}\rho^{n}\|^2-\|\delta_{t}\rho^{n-1}\|^2}{2\tau}+\frac{b_{*}}{4\tau}(\|\rho^{n+1}\|^2_1-\|\rho^{n-1}\|^2_1) = \bigl( \sum_{j=1}^7 F_j,\delta_{\hat{t}}\rho^{n}\bigr),
\end{equation}
where $ F_1- F_7$ denote the real parts of the seven terms on the right‑hand side of \eqref{eq:total_identity}.

For each term on the right-hand side of \eqref{3.10},  we bound them as follows:
\begin{align*}
\begin{split}
 |(F_1,\delta_{\hat{t}}\rho^{n})|=&\Re(\partial^2_t u^{n}-\delta^2_{t}u^n,\delta_{\hat{t}}\rho^{n})
\leq C\|\partial^4_t u\|^2_{L^2(t^{n-1},t^{n+1};L^{2})}\tau^3+\|\delta_{\hat{t}}\rho^{n}\|^2,
\end{split}\\
\begin{split}
 |(F_2,\delta_{\hat{t}}\rho^{n})|=&\Re(\delta^2_{t}\eta^n,\delta_{\hat{t}}\rho^{n})\leq C\|\delta^2_{t}\eta^{n}\|^2+\|\delta_{\hat{t}}\rho^{n}\|^2\\
&\leq C\|\partial^2_{t}\eta^{n}\|^2+\tau^4+\|\delta_{\hat{t}}\rho^{n}\|^2\\
&\leq C\|\partial^2_{t}u^{n}\|^2H^8+\tau^4+\|\delta_{\hat{t}}\rho^{n}\|^2,
\end{split}\\
\begin{split}
 |(F_3,\delta_{\hat{t}}\rho^{n})|=&\Im(\partial_t u^{n}-\delta_{\hat{t}}u^n,\delta_{\hat{t}}\rho^{n})
\leq C\|\partial^3_t u\|^2_{L^2(t^{n-1},t^{n+1};L^2)}\tau^3+\|\delta_{\hat{t}}\rho^{n}\|^2,
\end{split}\\
\begin{split}
|(F_4,\delta_{\hat{t}}\rho^{n})|=&\Im|(\delta_{\hat{t}}\eta^{n},\rho^{n})|\leq C \frac{1}{\tau}\int^{t^{n+1}}_{t^{n-1}}\|\partial_t\eta\|dt\cdot\|\delta_{\hat{t}}\rho^{n}\|
\leq C H^8 +\|\delta_{\hat{t}}\rho^{n}\|^2,
\end{split}\\
\begin{split}
|(F_5,\delta_{\hat{t}}\rho^{n})|=&\Re|(V\eta^{n}, \delta_{\hat{t}}\rho^{n})|
\leq C\|\eta^{n}\|^2+\|\delta_{\hat{t}}\rho^{n}\|^2\leq CH^{8}+\|\delta_{\hat{t}}\rho^{n}\|^2,
\end{split}\\
\begin{split}
|(F_6,\delta_{\hat{t}}\rho^{n})|=&\Re(r^n,\delta_{\hat{t}}\rho^{n})\leq C(\tau^4+\|\delta_{\hat{t}}\rho^{n}\|^2),
\end{split}\\
\begin{split}
|(F_7,\delta_{\hat{t}}\rho^{n})|=&\Re(f(|{u}^{n}|^2){u}^{n}-\tilde{f}(|{u}^{n+1}_H|^2,|{u}^{n-1}_H|^2)\bar{u}^{n}_H,\delta_{\hat{t}}\rho^{n})\\
=&\Re|(f(|{u}^{n}|^2)({u}^{n}-\frac{1}{2}({u}^{n+1}+{u}^{n-1}))+\frac{1}{2}(f(|{u}^{n}|^2)-\tilde{f}(|{u}^{n+1}|^2,|{u}^{n-1}|^2))({u}^{n+1}+{u}^{n-1})\\
&+\frac{1}{2}(\tilde{f}(|{u}^{n+1}|^2,|{u}^{n-1}|^2)({u}^{n+1}+{u}^{n-1})-\tilde{f}(|{u}^{n+1}_H|^2,|{u}^{n-1}_H|^2)({u}^{n+1}_H+{u}^{n-1}_H),\delta_{\hat{t}}\rho^{n})|\\
&=|(I_1+I_2+I_3,\delta_{\hat{t}}\rho^{n})|.
\end{split}
\end{align*}
\begin{align*}
  I_1&=\frac{1}{2}|f(|{u}^{n}|^2)(2{u}^{n}-({u}^{n+1}+{u}^{n-1})|\leq C\tau^2\|\partial^2_tu^n\|,\\
  I_2&= \frac{|f'(t_{*})|}{4}|[(2|{u}^{n}|^2-|{u}^{n+1}|^2-|{u}^{n-1}|^2)]({u}^{n+1}+{u}^{n-1})|\\
     &\leq C|2{u}^{n}-{u}^{n+1}-{u}^{n-1}|\\
        &\leq C\tau^2\|\partial^2_tu^n\|,
\end{align*}
where we use $\|u^{n}\|_{L^{\infty}}\leq C$ and
\begin{align*}
  f(|{u}^{n}|^2)&=\int^{1}_{0}f(|{u}^{n}|^2)ds \\
   & =\int^{1}_{0}f((1-s)|{u}^{n+1}|^2+s |{u}^{n-1}|^2)ds+\int^{1}_{0}f'(t_{*})[|{u}^{n}|^2-((1-s)|{u}^{n+1}|^2+s |{u}^{n-1}|^2)]ds\\
   & =\tilde{f}(|{u}^{n+1}|^2,|{u}^{n-1}|^2)+\int^{1}_{0}f'(t_{*})[|{u}^{n}|^2-((1-s)|{u}^{n+1}|^2+s |{u}^{n-1}|^2)]ds\\
   &=\tilde{f}(|{u}^{n+1}|^2,|{u}^{n-1}|^2)+\frac{f'(t_{*})}{2}[2|{u}^{n}|^2-|{u}^{n+1}|^2-|{u}^{n-1}|^2],
\end{align*}
here $t_{*}$ is between $|{u}^{n}|^2$ and $\frac{|{u}^{n+1}|^2+|{u}^{n-1}|^2}{2}$.
\begin{align*}
  I_3&= \frac{1}{2}\tilde{f}(|{u}^{n+1}|^2,|{u}^{n-1}|^2)({\rho}^{n+1}+{\eta}^{n+1}+{\rho}^{n-1}+{\eta}^{n-1})\\
  &+\frac{1}{2}[\tilde{f}(|{u}^{n+1}|^2,|{u}^{n-1}|^2)-\tilde{f}(|{u}^{n+1}_H|^2,|{u}^{n-1}_H|^2)]({u}^{n+1}_H+{u}^{n-1}_H) \\
  &=I_{31}+I_{32}.
\end{align*}
\begin{equation*}
  I_{31}\leq C ({\rho}^{n+1}+{\eta}^{n+1}+{\rho}^{n-1}+{\eta}^{n-1})\leq C({\rho}^{n+1}+{\rho}^{n-1}+H^4).
\end{equation*}
As for $I_{32}$,
\begin{align*}
  &\tilde{f}(|{u}^{n+1}|^2,|{u}^{n-1}|^2)-\tilde{f}(|{u}^{n+1}_H|^2,|{u}^{n-1}_H|^2)=\int^{1}_{0}(f(t_1)-f(t_2))ds \\
  & =\int^{1}_{0}f'(t^{*})(t_1-t_2)ds =\frac{f'(t^{*})}{2}[(|{u}^{n+1}|^2+|{u}^{n-1}|^2)-(|{u}^{n+1}_H|^2+|{u}^{n-1}_H|^2)],
\end{align*}
where $t_1=(1-s)|{u}^{n+1}|^2+s |{u}^{n-1}|^2$, $t_2=(1-s)|{u}^{n+1}_H|^2+s |{u}^{n-1}_H|^2$, $t^{*}$ is between $t_1$ and $t_2$.
\begin{align*}
  I_{32} &=\frac{|f'(t^{*})|}{4}[(|{u}^{n+1}|^2+|{u}^{n-1}|^2)-(|{u}^{n+1}_H|^2+|{u}^{n-1}_H|^2)]({u}^{n+1}_H+{u}^{n-1}_H) \\
  &\leq\frac{|f'(t^{*})|}{4}(|{u}^{n+1}|+|{u}^{n-1}|+|{u}^{n+1}_H|+|{u}^{n-1}_H|)(|{\rho}^{n+1}|+|{\rho}^{n-1}|+|{\eta}^{n+1}|+|{\eta}^{n-1}|)({u}^{n+1}_H+{u}^{n-1}_H)\\
  &\leq C ({\rho}^{n+1}+{\rho}^{n-1}+|{\rho}^{n+1}|+|{\rho}^{n-1}|+H^4).
\end{align*}
Thus, we have
\begin{equation*}
  I_3\leq C ({\rho}^{n+1}+{\rho}^{n-1}+|{\rho}^{n+1}|+|{\rho}^{n-1}|+H^4).
\end{equation*}
Then, we obtain
\begin{equation}\label{3.17}
  |(F_7,\delta_{\hat{t}}\rho^{n})|\leq C(\|\delta_{\hat{t}}\rho^{n}\|^2+\|{\rho}^{n+1}\|^2+\|{\rho}^{n-1}\|^2+H^{8}+\tau^4).
\end{equation}
From Lemma \ref{lem2.5}, we obtain
 \begin{equation}\label{3.178}
 \|\rho^{n+1}\|^2-\|\rho^{n-1}\|^2\leq 2\tau[\frac{(\|\delta_t\rho^n\|^2+\|\delta_t\rho^{n-1}\|^2)}{2}+\frac{(|\rho^{n+1}|^2+|\rho^{n-1}|^2)}{2}].
 \end{equation}
Furthermore, from \eqref{3.10}-\eqref{3.178}, we have
\begin{align}\label{3.18}
 \begin{split}
&\frac{1}{2\tau}(\|\delta_t\rho^{n}\|^2-\|\delta_t\rho^{n-1}\|^2)+\frac{b_{*}}{4\tau}(\|\rho^{n+1}\|^2_1-\|\rho^{n-1}\|^2_1)+\frac{1}{2\tau}(\|\rho^{n+1}\|^2-\|\rho^{n-1}\|^2)\\
&\leq C(H^{8}+\tau^4+\|\partial^4_t u\|^2_{L^2(t^{n-1},t^{n+1};L^{2})}\tau^3+\|\partial^3_t u\|^2_{L^2(t^{n-1},t^{n+1};L^2)}\tau^3\\
&+\|\delta_t\rho^n\|^2+\|\delta_t\rho^{n-1}\|^2+\|\rho^{n+1}\|^2+\|\rho^{n-1}\|^2),
 \end{split}
\end{align}
 Multiplying $2\tau$ on both sides of the above inequality and summing up for $n$ and choosing
$R_Hu_0 = u_0$, we obtain
\begin{align}
\begin{split}
&\|\delta_t\rho^{n}\|^2+\frac{b_{*}}{2}(\|\rho^{n+1}\|^2_1+\|\rho^{n}\|^2_1)+(\|\rho^{n+1}\|^2+\|\rho^{n}\|^2)\\
&\leq C(\tau^4+H^{8})+C\tau\sum^{n}_{k=1}[\|\delta_t\rho^k\|^2+(\|\rho^{k+1}\|^2_1+\|\rho^{k}\|^2_1)+(\|\rho^{k+1}\|^2+\|\rho^{k}\|^2)].
 \end{split}
\end{align}
If $\tau$ is small enough, it follows from Lemma \ref{lem2.2} that
 \begin{equation} \label{3.19}
 \|\delta_t\rho^{n}\|^2+\|\rho^{n}\|^2_1+\|\rho^{n}\|^2 \leq C(H^{8}+\tau^4).
 \end{equation}
 Then we have
  \begin{equation} \label{3.20}
\|\rho^{n}\|^2_1 \leq C(H^{8}+\tau^4).
 \end{equation}
Finally, applying the Sobolev embedding theorem and \eqref{3.20}, we arrive at
  \begin{equation} \label{3.21}
\|\rho^{n}\|^2_{L^p} \leq C(H^{8}+\tau^4).
 \end{equation}
\end{proof}
\end{lemma}

\subsection{Proof of Theorem \ref{main_thm_2.1}}
With the above estimates, the proof of Theorem~\ref{main_thm_2.1} is now complete.
\begin{proof}
A combination of Lemma \ref{lem3.1} and Lemma \ref{lem3.5} yields the following error estimate
\begin{equation}
 \max_{1\leq n \leq N}\|u^n-u^n_H\|_{L^p}\leq C(\tau^2+H^{4}).
\end{equation}  
\end{proof}

\section{Conclusion}
This paper proposes a multiscale method based on the Localized Orthogonal Decomposition (LOD) for the NLS equation with wave operator. We establish the existence and uniqueness of the semi-discrete solution and prove that the scheme satisfies a discrete energy conservation law. Furthermore, we rigorously derive optimal superconvergence error estimates in the \(L^2\) and \(H^1\) norms without time-step restrictions. Both theoretical analysis and numerical experiments demonstrate that the proposed LOD method achieves fourth-order convergence in the \(L^2\) norm and third-order in the \(H^1\) norm for the spatial discretization.


\subsection*{Funding} Hanzhang Hu is partially supported by the Guangdong Basic and Applied Basic Research Foundation (2024A1515012430,2026A1515010590). Lei Zhang is partially supported by the National Natural Science Foundation of China (Grant No. 12271360), the Shanghai Municipal Science and Technology Project 23JC1402300, and the Fundamental Research Funds for the Central Universities.

\subsection*{Data Availability}
Data sharing is not applicable to this paper as no datasets were generated or analyzed during
the current study.

\section*{Declarations}
%
\subsection*{Conflict of interest}
The authors declare that they have no conflict of interest.

\bibliographystyle{plain}
\bibliography{ref.bib}
\end{document}